\documentclass[oneside,11pt]{article}
\usepackage{amsmath,amssymb,amsthm,amsfonts}

\usepackage{authblk}          % author / affiliation
\usepackage[mathscr]{eucal}   % script math font

\usepackage[numbers,sort&compress]{natbib}

\usepackage[
  colorlinks=true,
  linkcolor=blue,
  citecolor=blue,
  urlcolor=blue,
  anchorcolor=blue,
  pdfborder={0 0 0},
  plainpages=false,
  pdfpagelabels
]{hyperref}

\theoremstyle{plain}
\newtheorem{theorem}{Theorem}[section]
\newtheorem{proposition}[theorem]{Proposition}

\newtheorem{corollary}[theorem]{Corollary}

\theoremstyle{definition}
\newtheorem{remark}[theorem]{Remark}
\theoremstyle{plain}
\newcommand{\bracenom}[2]{\genfrac\{\}{0pt}{}{#1}{#2}}
\newtheorem{theoremA}{Theorem}

\let\s\sigma          % sigma_n, sigma_n(t)
\def\D{\Delta}        % forward difference operator
\date{}

\begin{document}

	\begin{NoHyper}
		\let\thefootnote\relax\footnotetext{
			\\
			{\rm E-mail}: stj@inha.ac.kr}  
   %\\
%			{${}^2$\rm E-mail}: 22151060@inha.edu}
	\end{NoHyper}
	
	\author{Sangtae Jeong}
	\affil{Department of Mathematics, Inha University, Incheon 22212, Republic of Korea}

	\title {{\bf On a Question of Lehmer concerning the Comtet Numbers}}
	
	\maketitle

	\begin{abstract}
The Comtet numbers $b(n,k)$ and $B(n,k)$ of the first and second kind arise from the powers of
$(1+x)\log(1+x)$ and of its compositional inverse, respectively. By interpreting both families as
special values of weighted Stirling polynomials, we answer Lehmer's question of whether $B(n,k)$
satisfies a recurrence with a fixed number of terms by proving the four-term recurrence
\[
  k\,B(n+1,k+1)=B(n,k-1)-(n-k)\,B(n,k)-B(n+1,k).
\]
The
corresponding relation for the \emph{unsigned} array was conjectured by M.~Kurkov, but not proved, in
OEIS A354794. We further derive explicit formulas, convolution identities, and differential--difference
relations for the Touchard-type polynomials attached to the two families.
	\end{abstract}

\section{Introduction}

In 1974, Comtet \cite{Comt} introduced the numbers $b(n,k)$, defined for $k\ge1$ by
\begin{equation}
    \sum_{n\ge1} b(n,k)\frac{x^n}{n!}
    =
    \frac{\bigl[(1+x)\log(1+x)\bigr]^k}{k!}.
    \label{EGF-bnk}
\end{equation}
These numbers arise naturally in the formula for the $n$th derivative of $x^x$:
\[
D^n x^x
=
x^x
\sum_{j=0}^{n}
(\log x)^j
\binom{n}{j}
\sum_{h=0}^{n-j}
b(n-j,n-j-h)x^{-h};
\]
see also Gould \cite{Gou}, who studied a set of polynomials attached to these derivatives.

In 1985, Lehmer \cite{Leh} introduced a second family of numbers $B(n,k)$, again for $k\ge1$, by means of the exponential generating function
\begin{equation}
    \sum_{n\ge1}
    B(n,k)\frac{x^n}{n!}
    =
    \frac{\psi(x)^k}{k!},
    \label{EGF-Bnk}
\end{equation}
where
\begin{equation}
    \psi(x)
    =
    \sum_{\nu\ge1}
    (-1)^{\nu-1}
    (\nu-1)^{\nu-1}
    \frac{x^\nu}{\nu!},
\end{equation}
which is the compositional inverse of the function $(1+x)\log(1+x)$.

Throughout this paper we adopt the conventions
\[
b(n,0)=B(n,0)=\delta_{n,0},
\qquad
b(n,k)=B(n,k)=0
\quad (k>n),
\]
where $\delta_{n,0}$ denotes the Kronecker delta.

Lehmer called the numbers $b(n,k)$ and $B(n,k)$ the \emph{Comtet numbers of the first kind} and the \emph{Comtet numbers of the second kind}, respectively, in analogy with the classical Stirling numbers of the first and second kinds.

As with the Stirling numbers, the two Comtet families are orthogonal:
\[
\sum_{k=1}^{n}
b(n,k)B(k,r)
=
\delta_{n,r}
=
\sum_{k=1}^{n}
B(n,k)b(k,r).
\]

Lehmer also investigated the Taylor expansion of $(1+x)^{1+x}$, proving that
\begin{equation}
(1+x)^{1+x}
=
\sum_{n\ge0}
\sigma_n
\frac{x^n}{n!},
\label{def-L-x2x}
\end{equation}
where $\sigma_0=\sigma_1=1$ and, for $n\ge2$,
\begin{equation}
\sigma_n
=
n!
\sum_{k=1}^{\lfloor n/2\rfloor}
\frac{
s(n-k-1,k)+s(n-k-1,k-1)
}
{(n-k)!},
\label{def-sigma-n}
\end{equation}
with $s(n,k)$ denoting the signed Stirling numbers of the first kind, defined later in \eqref{defsnk}.

For the numbers of the first kind, Comtet \cite{Comt} established the four-term recurrence
\begin{equation}
b(n+1,k)
=
nb(n-1,k-1)
+b(n,k-1)
-(n-k)b(n,k),
\label{RR-b-intro}
\end{equation}
which Lehmer \cite[Equation (23)]{Leh} later rederived directly from the exponential generating function.

Lehmer \cite[p.~473]{Leh} subsequently asked whether the numbers $B(n,k)$ of the second kind also satisfy a recurrence involving only a \emph{fixed} number of terms. To the best of our knowledge, this question has remained unanswered. Our first main theorem answers it in the affirmative.

\begin{theoremA}[Theorem~\ref{thm-main-B}]\label{thmA}
For all $n\ge k\ge1$,
\begin{equation}
kB(n+1,k+1)
=
B(n,k-1)
-(n-k)B(n,k)
-B(n+1,k).
\label{RR-B-intro}
\end{equation}
\end{theoremA}

The triangle $B(n,k)$ was tabulated by Lehmer \cite{Leh} and appears in the
\emph{On-Line Encyclopedia of Integer Sequences}
as sequence A039621 \cite{OEIS39621}, together with its unsigned counterpart A354794 \cite{OEIS354794}. Writing $T(n,k)=(-1)^{\,n-k}B(n,k)$ for the unsigned array, the recurrence \eqref{RR-B-intro} is equivalent to
\begin{equation}\label{RR-T-intro}
T(n+1,k)
=
kT(n+1,k+1)+T(n,k-1)+(n-k)T(n,k),
\end{equation}
in which every term is nonnegative, reflecting a combinatorial interpretation of $T(n,k)$ recorded in A354794 and discussed in Section~\ref{sec:Bnk}; the recurrence \eqref{RR-T-intro} was conjectured in A354794 by M.~Kurkov (20 April 2026) but left unproved. On the same page, P.~Luschny (8 May 2026) observed a Lambert--$W$ row-polynomial convolution and remarked, without proof, that the corresponding coefficient recurrence follows from it.

The aim of this paper is to give a unified account of the Comtet numbers of both kinds. Our principal result answers Lehmer's question, establishing the fixed-term recurrence of Theorem~\ref{thmA} for $B(n,k)$ and thereby resolving the conjectured recurrence \eqref{RR-T-intro} for $T(n,k)$. Beyond this, we investigate the combinatorial properties of the two families and derive a range of identities for them---explicit formulas, convolution identities, and further recurrences, together with differential--difference relations for the associated Touchard-type polynomials and partial Bell polynomial representations.

Our approach places both families within the framework of weighted Stirling polynomials. Specifically, $b(n,k)$ is realized as a special value of the weighted Stirling polynomial $s_{n,k}(x)$ of the first kind, while $B(n,k)$ is obtained as a special value of the corresponding weighted Stirling polynomial $G_{n,k}(x)$ of the second kind. The recurrence in Theorem~\ref{thmA} follows from two elementary identities satisfied by $G_{n,k}(x)$. The same correspondence also yields a second proof of Comtet's recurrence \eqref{RR-b-intro}, together with explicit formulas and convolution identities for both families.

The remainder of the paper is organized as follows. Section~\ref{sec:wsp} reviews the weighted Stirling polynomials required in the sequel; complete proofs of the identities quoted there are deferred to Appendix~\ref{app:proofs}, so that the paper is self-contained. Section~\ref{sec:comtet} develops the corresponding properties of the Comtet numbers. Theorem~\ref{thmA} is proved in Section~\ref{sec:rec}. In Section~\ref{sec:bell}, we study the Touchard-type polynomials $\sigma_n(t)$ and $\Sigma_n(t)$ attached to the two families and recover the partial Bell polynomial representations of the Comtet numbers of both kinds.

\section{Weighted Stirling polynomials and numbers of two kinds}
\label{sec:wsp}

In this section we briefly review the properties of the weighted Stirling polynomials and numbers of the first and second kinds that will be used throughout the paper. We also collect several auxiliary facts concerning the coefficients $c(n,k)$ arising in the expansion of the binomial polynomial $\binom{x+1}{n}$ in powers of $x$.

\subsection{Signed weighted Stirling polynomials and numbers of the first kind}

The signed Stirling numbers of the first kind are defined by the expansion of the falling factorial
\begin{equation}
(y)_n
=
\sum_{k=0}^{n}
s(n,k)y^k,
\label{defsnk}
\end{equation}
where
\[
(y)_n=y(y-1)\cdots(y-n+1),\qquad (y)_0:=1.
\]

More generally, expanding the two-variable falling factorial $(x+y)_n$ with respect to the variable $y$ gives
\begin{equation}
(x+y)_n
=
\sum_{k=0}^{n}
s_{n,k}(x)y^k,
\label{defsnkx}
\end{equation}
which uniquely defines a polynomial $s_{n,k}(x)\in\mathbb{C}[x]$ for each pair $(n,k)$.
Setting $x=0$ in \eqref{defsnkx} and comparing with \eqref{defsnk} immediately yields
\[
s_{n,k}(0)=s(n,k),
\qquad
s_{n,0}(x)=(x)_n.
\]
Thus the polynomials $s_{n,k}(x)$ constitute a natural polynomial extension of the signed Stirling numbers of the first kind, and will be referred to as the \emph{signed weighted Stirling polynomials of the first kind}.

We conclude this subsection by collecting the properties of $s_{n,k}(x)$ that will be needed in the subsequent sections.
\begin{theorem}\label{BPsnkx}
(i) Addition formulas:
\begin{eqnarray}
    s_{n,k}(x+y) &=& \sum_{j=k}^{n} \binom{j}{k} s_{n,j}(y) x^{j-k};\label{AFsnkx} \\
    s_{n,k}(x+y) &=& \sum_{j=k}^{n} \binom{n}{j} s_{j,k}(y) (x)_{n-j}.\label{AFsnkx2}
\end{eqnarray}
(ii) Explicit expressions:
\begin{eqnarray}
    s_{n,k}(x) &=&\sum_{j=k}^{n} \binom{j}{k}s(n,j)x^{j-k};\label{EEsnk1} \\
    s_{n,k}(y) &=&\sum_{j=k}^{n} \binom{n}{j}s(j,k)(y)_{n-j}.\label{EEsnk2}
\end{eqnarray}
(iii) Recurrence relations:
\begin{equation} \label{RRsnkx}
    s_{n+1,k}(x) = (x-n) s_{n,k}(x)+ s_{n,k-1}(x).
\end{equation}
\begin{equation} \label{RRsnk}
    s(n+1,k) = -n\,s(n,k)+ s(n,k-1).
\end{equation}
(iv) Exponential generating function:
 \begin{equation}\label{EGFsnkx}
    \sum_{n\geq k} s_{n,k}(x) \frac{t^n}{n!}=(1+t)^x \frac{(\ln(1+t))^k}{k!}.
\end{equation}
(v) Double generating function:
\begin{equation} \label{DGFsnkx}
    \sum_{k\geq 0}\sum_{n\geq k} s_{n,k}(x) \frac{t^n}{n!} y^k=(1+t)^{x+y}.
\end{equation}
(vi) Convolution identity:
\begin{equation}  \label{CFsnkx}
  \binom{k+j}{k}s_{n,k+j}(x+y) =  \sum_{r=k}^{n-j}  \binom{n}{r}s_{r,k}(x)s_{n-r, j}(y).
\end{equation}
\end{theorem}
\begin{proof}
Complete proofs are given in Appendix~\ref{app:first}; each identity is a short consequence of the
exponential generating function \eqref{EGFsnkx}, which is itself read off from the definition
\eqref{defsnkx}. These identities also appear in \cite{J-St1}, and several of them, in related forms,
in the earlier works of Carlitz
\cite{Car80a,Car80b}, Broder \cite{Bro}, and Koutras \cite{Kou}.
\end{proof}

The next theorem states the differentiation and difference formulas for $s_{n,k}(x)$.
\begin{theorem}\label{DG-SF-snkx}
(i) Differentiation formula:
\begin{equation} \label{DFsnkx}
    D s_{n,k}(x) =(k+1)s_{n,k+1}(x);\quad \quad  D^r s_{n,k}(x) =(k+1)^{\overline{r}}s_{n,k+r}(x),
\end{equation}
where $D=\frac{d}{dx}$ and $k^{\overline{r}}$ is the rising factorial.

(ii) Difference formula:
\begin{equation} \label{SFsnkx}
    \D s_{n,k}(x) =ns_{n-1,k}(x);\quad \quad  \D^r s_{n,k}(x) =(n)_rs_{n-r,k}(x),
\end{equation}
where $\D f(x)=f(x+1)-f(x)$ and $(n)_r$ is the falling factorial.
\end{theorem}
\begin{proof}
Both formulas follow from the defining relation \eqref{defsnkx}. Since $(x+y)_n$ depends only on
$x+y$, the operators $D=\partial_x$ and $\partial_y$ agree on it; as $\partial_y$ raises the lower
index, $D s_{n,k}(x)=(k+1)s_{n,k+1}(x)$, and iteration gives the $D^r$ form. Likewise
$\D_x(x+y)_n=n\,(x+y)_{n-1}$ gives $\D s_{n,k}(x)=n\,s_{n-1,k}(x)$, and iteration gives the $\D^r$ form.
\end{proof}

We close this subsection with a further recurrence for $s_{n,k}(x)$, its value at $x=1$, and the
classical inversion formula for the Stirling numbers of the first kind.
\begin{proposition}
For $k\ge1$,
\begin{eqnarray}
 s_{n,k}(x) &=& x\,s_{n-1,k}(x-1) + s_{n-1,k-1}(x-1);\label{IFHRRsnk-b}\\
 s_{n,k}(1) &=& s(n-1,k) + s(n-1,k-1).\label{IFHRRsnk-c}
\end{eqnarray}
Moreover, for $k\ge0$, the classical inversion formula
\begin{equation}\label{IFHRRsnk}
 s_{n,k}(-1)=\sum_{j=k}^n (-1)^{j-k}\binom{j}{k} s(n,j)= s(n+1,k+1)
\end{equation}
holds.
\end{proposition}
\begin{proof}
To prove \eqref{IFHRRsnk-b}, we start with the falling-factorial identity
\[
 s_{n,0}(x)=x\,s_{n-1,0}(x-1),
\]
that is, $(x)_{n}=x\,(x-1)_{n-1}$. Applying the Leibniz rule for the derivative $D^{k}$ to this
identity, together with \eqref{DFsnkx}, we obtain
\[
   k!\, s_{n,k}(x)=x\,k!\, s_{n-1,k}(x-1)+k!\, s_{n-1,k-1}(x-1),
\]
and dividing by $k!$ gives \eqref{IFHRRsnk-b}. Setting $x=1$ in \eqref{IFHRRsnk-b} and using
$s_{n-1,k}(0)=s(n-1,k)$ and $s_{n-1,k-1}(0)=s(n-1,k-1)$ yields \eqref{IFHRRsnk-c}. Finally, setting
$x=0$ in \eqref{IFHRRsnk-b} and using $s_{n,k}(0)=s(n,k)$ gives $s(n,k)=s_{n-1,k-1}(-1)$; replacing
$(n,k)$ by $(n+1,k+1)$ then yields $s(n+1,k+1)=s_{n,k}(-1)$, while the summation expression follows
from \eqref{EEsnk1} at $x=-1$. This is the classical inversion formula; see
\cite[p.~265, Eq.~(6.21)]{GKP} and \cite[Theorem~8.8, Eq.~(8.30)]{Ch}.
\end{proof}

\subsection{Unsigned weighted Stirling polynomials and numbers of the first kind}

As with the polynomials $s_{n,k}(x)$, define the unsigned weighted Stirling polynomials of the first kind by
\begin{equation}\label{defrnkx}
(x+y)^{\overline n}
=
\sum_{k=0}^{n}
r_{n,k}(x)y^k,
\end{equation}
where
\[
(x+y)^{\overline n}
=(x+y)(x+y+1)\cdots(x+y+n-1),\qquad (x+y)_0:=1.
\]
The coefficients $r_{n,k}(x)\in\mathbb{C}[x]$ are uniquely determined. Setting $x=0$ gives
\begin{equation}\label{defrnk}
y^{\overline n}
=
\sum_{k=0}^{n}
r(n,k)y^k,
\end{equation}
so that
\[
r(n,k)=r_{n,k}(0),
\qquad
r_{n,0}(x)=x^{\overline n}.
\]
Thus $\{r_{n,k}(x)\}$ is a polynomial extension of the unsigned Stirling numbers of the first kind.

Replacing $(x,y)$ by $(-x,-y)$ in \eqref{defsnkx} gives
\[
(x+y)^{\overline n}
=
\sum_{k=0}^{n}
(-1)^{\,n+k}s_{n,k}(-x)y^k.
\]
Comparison with \eqref{defrnkx} yields
\begin{equation}\label{RSR}
r_{n,k}(x)
=
(-1)^{\,n+k}s_{n,k}(-x),
\end{equation}
and hence
\begin{equation}\label{R-rnksnk}
r(n,k)
=
(-1)^{\,n+k}s(n,k).
\end{equation}

The polynomials $r_{n,k}(x)$ satisfy identities entirely analogous to those of the signed weighted
Stirling polynomials in Theorem~\ref{BPsnkx}.
\begin{theorem}\label{BPrnkx}
(i) Addition formulas:
\begin{eqnarray}
    r_{n,k}(x+y) &=&\sum_{j=k}^{n} \binom{j}{k} r_{n,j}(y) x^{j-k};\label{AFrnkx} \\
    r_{n,k}(x+y) &=&\sum_{j=k}^{n} \binom{n}{j} r_{j,k}(y) x^{\overline{n-j}}.\label{AFrnkx2}
\end{eqnarray}
(ii) Explicit expression:
\begin{equation} \label{EErnk1}
    r_{n,k}(x) =\sum_{j=k}^{n} \binom{j}{k}r(n,j)x^{j-k}.
\end{equation}
(iii) Recurrence relations:
\begin{equation} \label{RRrnkx}
    r_{n+1,k}(x) = (n+x) r_{n,k}(x)+ r_{n,k-1}(x).
\end{equation}
\begin{equation} \label{RRrnk}
     r({n+1,k}) = n r({n,k})+ r({n,k-1}).
\end{equation}
(iv) Exponential generating function:
 \begin{equation}\label{EGFrnkx}
    \sum_{n\geq k} r_{n,k}(x) \frac{t^n}{n!}=(1-t)^{-x} \frac{(-\ln(1-t))^k}{k!}.
\end{equation}
(v) Double generating function:
\begin{equation} \label{DGFrnkx}
    \sum_{k\geq 0}\sum_{n\geq k} r_{n,k}(x) \frac{t^n}{n!} y^k=(1-t)^{-x-y}.
\end{equation}
(vi) Convolution identity:
\begin{equation} \label{CFrnkx}
  \binom{k+j}{k}r_{n,k+j}(x+y) =  \sum_{r=k}^{n-j}  \binom{n}{r}r_{r,k}(x)r_{n-r, j}(y).
\end{equation}
\end{theorem}
\begin{proof}
By \eqref{RSR}, $r_{n,k}(x)=(-1)^{\,n+k}s_{n,k}(-x)$, so each of (i)--(vi) is the image of the
corresponding identity of Theorem~\ref{BPsnkx} under $x\mapsto-x$; we record the point that needs
care in each case. For the generating functions the substitution is accompanied by $t\mapsto-t$, so
\eqref{EGFsnkx} gives
\[
  \sum_{n}r_{n,k}(x)\frac{t^{n}}{n!}
  =(-1)^{k}\sum_{n}s_{n,k}(-x)\frac{(-t)^{n}}{n!}
  =(1-t)^{-x}\frac{\bigl(-\ln(1-t)\bigr)^{k}}{k!},
\]
which is \eqref{EGFrnkx}, and multiplying by $y^{k}$ and summing over $k$ gives \eqref{DGFrnkx}.
Replacing $x$ by $-x$ in \eqref{RRsnkx} and multiplying by $(-1)^{\,n+1+k}$ gives \eqref{RRrnkx},
and $x=0$ then gives \eqref{RRrnk}. Applying $(x,y)\mapsto(-x,-y)$ to \eqref{AFsnkx} and
\eqref{AFsnkx2} and using \eqref{RSR}, the accumulated signs cancel to
$+1$, while $(-x)_{n-j}=(-1)^{\,n-j}x^{\overline{n-j}}$ turns the falling factorial into the rising
one; this yields \eqref{AFrnkx} and \eqref{AFrnkx2}, and $y=0$ in \eqref{AFrnkx} gives
\eqref{EErnk1}. Finally, the same substitution in \eqref{CFsnkx}, multiplied by $(-1)^{\,n+k+j}$,
gives \eqref{CFrnkx}.
\end{proof}

\begin{remark}\normalfont
By \eqref{RSR} and \eqref{R-rnksnk}, the constant terms $r(n,k)=(-1)^{\,n+k}s(n,k)$ are the unsigned
Stirling numbers of the first kind (Stirling cycle numbers), and $r_{n,k}(x)$ coincides with
Carlitz's weighted Stirling polynomial of the first kind,
$r_{n,k}(x)=R_1(n,k,x)=\overline{S}_1(n,k,x)+S_1(n,k)$ with $S_1(n,k)=r(n,k)$ \cite[Eq.~(1.4)]{Car80a}.
Its values at $x=r$ give Broder's $r$-Stirling numbers of the first kind,
$R_1(n,k,r)=\genfrac[]{0pt}{}{n+r}{k+r}_{r}$ \cite{Bro}. Thus studying $r_{n,k}(x)$ through $s_{n,k}(x)$
recovers and extends the properties of $R_1(n,k,x)$ obtained by Carlitz; the second-kind analogues
appear in Remark~\ref{rem:earlier}.
\end{remark}

The following theorem is the counterpart of Theorem~\ref{DG-SF-snkx}.
\begin{theorem}\label{DG-SF-rnkx}
(i) Differentiation formula:
\begin{equation} \label{DFrnkx}
    D r_{n,k}(x) =(k+1)r_{n,k+1}(x);\quad  D^r r_{n,k}(x) =(k+1)^{\overline{r}}r_{n,k+r}(x).
\end{equation}

(ii) Difference formula:
\begin{equation} \label{SFrnkx}
    \D r_{n,k}(x) =nr_{n-1,k}(x+1);\quad  \D^r r_{n,k}(x) =(n)_rr_{n-r,k}(x+r).
\end{equation}
\end{theorem}
\begin{proof}
Both formulas follow by applying \eqref{RSR} to \eqref{DFsnkx} and \eqref{SFsnkx}, respectively.
\end{proof}

We conclude with the connection formula relating the two families, together with two of its
consequences.
\begin{theorem}(Connection formula)
\begin{equation}\label{R-snkx-rnkx}
s_{n,k}(x)=r_{n,k}(x+1-n).
\end{equation}
\end{theorem}
\begin{proof}
We use the well-known relation between the falling and rising factorials,
\[
 s_{n,0}(x) =r_{n,0}(x+1-n),
\]
that is, $(x)_n=(x+1-n)^{\overline n}$. Applying the operator $\frac{D^k}{k!}$ to this identity and
using \eqref{DFsnkx} and \eqref{DFrnkx} yields \eqref{R-snkx-rnkx}.
\end{proof}
\begin{corollary}
\begin{eqnarray}
    r_{n,k}(-n) &=& s_{n,k}(-1)=s(n+1,k+1);\quad s_{n,k}(n) =r_{n,k}(1)= r(n+1,k+1). \label{my-rnkmn-EQ}\\
 r_{n,k}(1-n)&=& s(n,k);\quad s_{n,k}(n-1)= r(n,k).\label{my-rnkmn-EQ2}
\end{eqnarray}
\end{corollary}

We now present reflection formulas for $s_{n,k}(x)$ and $r_{n,k}(x)$.
\begin{theorem}(Reflection formula)
\begin{eqnarray}
s_{n,k}(\tfrac{n-1}{2}+x) &=& (-1)^{n+k}s_{n,k}(\tfrac{n-1}{2}-x);\label{RF-Feq-snkx2}\\
r_{n,k}(\tfrac{1-n}{2}+x) &=& (-1)^{n+k}r_{n,k}(\tfrac{1-n}{2}-x).\label{RF-Feq-snkx}
\end{eqnarray}
\end{theorem}
\begin{proof}
By \eqref{RSR} and the connection formula \eqref{R-snkx-rnkx},
\[
 s_{n,k}(x)=(-1)^{n+k}r_{n,k}(-x)=(-1)^{n+k}s_{n,k}(n-1-x);
\]
replacing $x$ by $\frac{n-1}{2}+x$ gives \eqref{RF-Feq-snkx2}. Likewise
\[
 r_{n,k}(x)=(-1)^{n+k}s_{n,k}(-x)=(-1)^{n+k}r_{n,k}(1-n-x),
\]
and replacing $x$ by $\frac{1-n}{2}+x$ gives \eqref{RF-Feq-snkx}.
\end{proof}

\subsection{Weighted Stirling polynomials and numbers of the second kind}

The four classes of weighted Stirling polynomials of the second kind introduced in
\cite{J-St1} are defined by
\begin{align}
(x+y)^n
&=
\sum_{k=0}^{n}
G_{n,k}(x)\binom{y}{k},
\label{defGnraa}\\
(x+y)^n
&=
\sum_{k=0}^{n}
J_{n,k}(x)\binom{y+k}{k},
\label{defJnraa}\\
(x+y)^n
&=
\sum_{k=0}^{n}
K_{n,k}(x)\binom{y+k-1}{k},
\label{defKnraa}\\
(x+y)^n
&=
\sum_{k=0}^{n}
H_{n,k}(x)\binom{y-1}{k}.
\label{defHnraa}
\end{align}

For each family we write $F(n,k)=F_{n,k}(x)\big|_{x=0}$ for its constant term, where $F_{n,k}(x)$
denotes any one of $G_{n,k}(x)$, $J_{n,k}(x)$, $K_{n,k}(x)$, $H_{n,k}(x)$. These constant
specializations recover, up to sign and a factorial factor, the classical Stirling numbers of the
second kind.

\begin{proposition}\label{prop:GJKH}
For $0\le k\le n$,
\begin{align}
G(n,k)
&=
k!S(n,k),
\label{defGnk}\\
J(n,k)
&=
(-1)^{\,n+k}k!S(n+1,k+1),
\label{defJnk}\\
K(n,k)
&=
(-1)^{\,n+k}k!S(n,k),
\label{defKnk}\\
H(n,k)
&=
k!S(n+1,k+1),
\label{defHnk}
\end{align}
\end{proposition}
\begin{proof}
See Appendix~\ref{app:second}.
\end{proof}

The remaining three families reduce to $G_{n,k}(x)$ as follows.
\begin{proposition}\label{prop:KJH-G}
For $0\le k\le n$,
\begin{align}
K_{n,k}(x)
&=
(-1)^{\,n+k}G_{n,k}(-x),
\label{R-KnrGnr}\\
J_{n,k}(x)
&=
(-1)^{\,n+k}G_{n,k}(1-x),
\label{R-JnrGnr}\\
H_{n,k}(x)
&=
G_{n,k}(1+x).
\label{R-HnrGnr}
\end{align}
\end{proposition}
\begin{proof}
See Appendix~\ref{app:second}.
\end{proof}

The polynomials $G_{n,k}(x)$ satisfy identities parallel to those for $s_{n,k}(x)$ established in
Theorem~\ref{BPsnkx}. Analogous identities hold for $H_{n,k}(x)$, $J_{n,k}(x)$, and $K_{n,k}(x)$. In
what follows, we focus on the family $G_{n,k}(x)$ and use these identities to study Comtet numbers.
\begin{theorem}\label{BPGnkx}
(i) Addition formulas:
\begin{eqnarray}
   G_{n,k}(x+y) &=& \sum_{j=k}^n \binom{n}{j} G_{j,k}(y)\, x^{n-j};\label{AFGnkx} \\
   G_{n,k}(x+y) &=& \sum_{j=k}^n G_{n,j}(y)\binom{x}{j-k}.\label{AFGnkx2}
\end{eqnarray}
(ii) Explicit expressions:
\begin{eqnarray}
   G_{n,k}(x) &=&\sum_{j=k}^n \binom{n}{j} G(j,k)\, x^{n-j};\label{defSP} \\
   G_{n,k}(x) &=&\sum_{\nu=0}^{k}(-1)^{k-\nu}\binom{k}{\nu}(x+\nu)^n.\label{Exp-Gnkx}
\end{eqnarray}
(iii) Recurrence relations:
\begin{equation}\label{RR-Gnkx}
     G_{n+1,k}(x)= (x+k)G_{n,k}(x) + kG_{n,k-1}(x).
\end{equation}
\begin{equation}\label{RR-Gnk}
     G(n+1,k)= kG(n,k) + kG(n,k-1).
\end{equation}
(iv) Exponential generating function:
\begin{equation}\label{stipol}
     \sum_{n\geq0}G_{n,k}(x)\frac{t^n}{n!}= e^{xt}\bigl(e^t-1\bigr)^k.
\end{equation}
(v) Double generating function:
\begin{equation}\label{DGFGnkx}
   \sum_{k\geq0}\sum_{n\geq k} G_{n,k}(x)\binom{y}{k}\frac{t^n}{n!}= e^{(x+y)t}.
\end{equation}
(vi) Convolution identity:
\begin{equation}\label{Con-Gnkx}
   G_{n,k+j}(x+y) = \sum_{r=k}^{n-j} \binom{n}{r} G_{r,k}(x) G_{n-r,j}(y).
\end{equation}
\end{theorem}
\begin{proof}
Complete proofs are given in Appendix~\ref{app:second}; each identity follows with some work
from the exponential generating function \eqref{stipol}, which is itself read off from the
definition \eqref{defGnraa}. These identities, too, occur as special cases in
\cite{J-St1}, with related forms appearing in the earlier works of Carlitz
\cite{Car80a,Car80b}, Broder \cite{Bro}, and Koutras \cite{Kou}.
\end{proof}

The next theorem states the differentiation and difference formulas for $G_{n,k}(x)$.
\begin{theorem}\label{DG-SF-Gnkx}
(i) Differentiation formula:
\begin{equation}\label{DFGnkx}
   D G_{n,k}(x) = n\,G_{n-1,k}(x);\quad\quad D^r G_{n,k}(x) = (n)_r\,G_{n-r,k}(x).
\end{equation}

(ii) Difference formula:
\begin{equation}\label{D-Gnr}
   \D G_{n,k}(x) = G_{n,k+1}(x);\quad\quad \D^r G_{n,k}(x) = G_{n,k+r}(x).
\end{equation}

\end{theorem}
\begin{proof}
Both formulas come from the generating function \eqref{stipol}: applying $D_x$ multiplies its
right-hand side by $t$, which yields \eqref{DFGnkx}, while applying $\D_x$ multiplies it by
$(e^t-1)$, which yields \eqref{D-Gnr}. Iteration gives the $D^r$ and $\D^r$ forms.
\end{proof}

Finally, $G_{n,k}(x)$ enjoys a reflection symmetry about the point $x=-k/2$, the second-kind
counterpart of the reflection formulas \eqref{RF-Feq-snkx2} and \eqref{RF-Feq-snkx}.
\begin{theorem}(Reflection formula)
\begin{equation}\label{RecF-Gnkx}
G_{n,k}\left(-\tfrac{k}{2}+x\right) = (-1)^{n+k}G_{n,k}\left(-\tfrac{k}{2}-x\right).
\end{equation}
\end{theorem}
\begin{proof}
Substituting $x\mapsto-x-k$ and $t\mapsto-t$ in the generating function \eqref{stipol}, and using
$1-e^{-t}=e^{-t}(e^t-1)$, we obtain
\[
\sum_{n\ge0}(-1)^{n+k}G_{n,k}(-x-k)\frac{t^n}{n!}
=(-1)^{k}e^{(x+k)t}\bigl(e^{-t}-1\bigr)^k
=e^{xt}\bigl(e^{t}-1\bigr)^k
=\sum_{n\ge0}G_{n,k}(x)\frac{t^n}{n!}.
\]
Comparing the coefficients of $t^n/n!$ gives $G_{n,k}(x)=(-1)^{n+k}G_{n,k}(-x-k)$; replacing $x$ by
$-\frac{k}{2}+x$ yields \eqref{RecF-Gnkx}.
\end{proof}

We now remark on several earlier studies concerning variants of the polynomials $G_{n,k}(x)$, and,
in the last item, of $s_{n,k}(x)$ and $r_{n,k}(x)$ as well.
\begin{remark}\label{rem:earlier}
{\rm
1. Applying the Newton interpolation formula \cite{Jor} to the monomial $f(x)=x^n$ yields the
representation \eqref{defGnraa}, with expansion coefficients
\begin{equation}\label{New-exp1}
G_{n,k}(x)=\Delta^k x^n=\sum_{j=0}^{k}(-1)^{k-j}\binom{k}{j}(x+j)^n,
\end{equation}
which coincides with the explicit expression \eqref{Exp-Gnkx}. This identity is classical; see
\cite[Example 1]{Ri58} or \cite[Equation 9.23]{QG}, where it is written as
\begin{equation}\label{New-exp2}
 \Delta^k x^n =\sum_{j=0}^n\binom{n}{j}k!S(j,k)x^{n-j}.
\end{equation}

2. Nielsen \cite{Nie} formulated $(x+y)^n$ using ${\mathcal A}_k^n(x)$, which is equivalent to
$\Delta^k x^n$, and used it to derive numerous formulas for Bernoulli and Euler polynomials.

3. Carlitz \cite{Car80a,Car80b} introduced polynomials $R(n,k,x)$ from a sum related to the weighted
Stirling numbers of the second kind $\overline{S}(n,k+1,x)$, namely
$R(n,k,x)=\overline{S}(n,k+1,x)+S(n,k)$. From \eqref{New-exp1} and \eqref{New-exp2} the relation
$G_{n,k}(x)=k!\,R(n,k,x)$ is immediate. Working with the normalization $G_{n,k}(x)$ rather than
$R(n,k,x)$ is advantageous: the factor $k!$ is precisely what renders the identities of
Theorem~\ref{BPGnkx} in their clean, uniform form---most visibly the exponential generating function
\eqref{stipol}. Carlitz deduced combinatorial and algebraic properties of $R(n,k,x)$.

4. Koutras \cite{Kou} used Newton's formula to define the non-central Stirling numbers of the second
kind with parameter $x$, $S_{x}(n,k)=\frac{1}{k!}\bigl[\D^k(t-x)^n\bigr]_{t=0}$; hence
$S_{-x}(n,k)=\frac{1}{k!}G_{n,k}(x)$; for the first kind, Koutras likewise defined $s(n,k;r)$ by
$(y-r)_n=\sum_{k}s(n,k;r)y^k$, so that $s_{n,k}(-r)=s(n,k;r)$ (see also
\cite[Chapter~8, Section~5]{Ch}).

5. Broder \cite{Bro} introduced the $r$-Stirling numbers $\bracenom{n}{k}_{r}$ of the second kind,
which count certain restricted partitions for positive integers $r$. In particular,
\cite[Theorem 22]{Bro} related them to the values of $G_{n,k}(x)$ at $x=r$:
\begin{equation}\label{Broder-Gnk}
   \bracenom{n+r}{k+r}_{r}=\frac{1}{k!}G_{n,k}(r).
\end{equation}
}
\end{remark}

\begin{remark}\normalfont
There are orthogonality relations between $G_{n,k}(x)$ and $s_{n,k}(x)$, and between $J_{n,k}(x)$ and
$r_{n,k}(x)$. In the former case, as Carlitz \cite{Car80a} did, one can use them to derive convolution identities of
various types. These are not pursued here; see \cite{J-St1} for details.
\end{remark}

Beyond these properties, the four families \eqref{defGnraa}--\eqref{defHnraa} also furnish binomial
bases of the vector space $\mathcal{P}_n(x)$ of polynomials of degree at most $n$. Setting the
polynomial variable to $0$ in each definition and renaming $y$ as $x$, the four families
\[
\left\{\binom{x}{k}\right\}_{k=0}^{n},\quad
\left\{\binom{x+k}{k}\right\}_{k=0}^{n},\quad
\left\{\binom{x+k-1}{k}\right\}_{k=0}^{n},\quad
\left\{\binom{x-1}{k}\right\}_{k=0}^{n},
\]
associated with $G$, $J$, $K$, and $H$ respectively, each form a basis of $\mathcal{P}_n(x)$, so every
polynomial in $\mathcal{P}_n(x)$ admits a unique expansion with respect to any one of them. The
following proposition records the coefficients in all four bases; the $J$- and $K$-expansions will be
used in Section~\ref{sec:Bnk}.
\begin{proposition}
For  any $f(x)=\sum_{j=0}^nc_j x^j \in \mathcal{P}_n(x),$
\begin{eqnarray}
    f(x) &=& \sum_{k=0}^n \sum_{j=k}^{n}
      c_jG(j,k)\binom{x}{k}=\sum_{k=0}^n \sum_{j=k}^{n}c_jS(j,k)(x)_k; \label{CoS1kx-G}\\
       f(x) &=& \sum_{k=0}^n \sum_{j=k}^{n}
      c_jJ(j,k)\binom{x+k}{k}=\sum_{k=0}^n \sum_{j=k}^{n}(-1)^{j+k}c_jS(j+1,k+1)(x+1)^{\overline k}; \label{CoS1kx3}\\
    f(x) &=& \sum_{k=0}^n \sum_{j=k}^{n}
      c_jK(j,k)\binom{x+k-1}{k}=\sum_{k=0}^n \sum_{j=k}^{n}(-1)^{j+k}c_jS(j,k)x^{\overline k}; \label{CoS1kx4}\\
    f(x) &=& \sum_{k=0}^n \sum_{j=k}^{n}
      c_jH(j,k)\binom{x-1}{k}=\sum_{k=0}^n \sum_{j=k}^{n}c_jS(j+1,k+1)(x-1)_k. \label{CoS1kx-H}
\end{eqnarray}
\end{proposition}
\begin{proof}
Setting $x=0$ in \eqref{defGnraa}--\eqref{defHnraa} and then renaming $y$ as $x$ gives the monomial
expansions
\[
\begin{aligned}
  x^j&=\sum_{k=0}^{j}G(j,k)\binom{x}{k}, &\quad
  x^j&=\sum_{k=0}^{j}J(j,k)\binom{x+k}{k},\\
  x^j&=\sum_{k=0}^{j}K(j,k)\binom{x+k-1}{k}, &\quad
  x^j&=\sum_{k=0}^{j}H(j,k)\binom{x-1}{k}.
\end{aligned}
\]
Substituting these into $f(x)=\sum_{j=0}^n c_jx^j$ and interchanging the order of summation yields the
first equality in each line. The second follows from \eqref{defGnk}--\eqref{defHnk} together with
$k!\binom{x}{k}=(x)_k$, $k!\binom{x+k}{k}=(x+1)^{\overline k}$, $k!\binom{x+k-1}{k}=x^{\overline k}$,
and $k!\binom{x-1}{k}=(x-1)_k$.
\end{proof}

\subsection{The coefficients $c(n,k)$}

We collect several elementary properties, used later, of the coefficients $c(n,k)$ arising from the
expansion of the binomial polynomial
\begin{equation}\label{def-cnx}
    c_n(x)=\binom{x+1}{n}
\end{equation}
in powers of $x$; that is, $c(n,k)$ is the coefficient of $x^k$ in $c_n(x)$.
Newton's binomial theorem gives
\begin{equation}\label{OGF-cnx}
(1+t)^{1+x}
=
\sum_{n\ge0}\binom{x+1}{n}t^n
=
\sum_{n\ge0}c_n(x)t^n,
\end{equation}
so that $c_n(x)$ is the coefficient of $t^n$ in the ordinary generating function of $(1+t)^{1+x}$.

Writing
\begin{equation}\label{cnx-Mono}
c_n(x)
=
\binom{x+1}{n}
=
\sum_{k=0}^{n}c(n,k)x^k,
\end{equation}
the coefficients $c(n,k)$ are obtained by differentiation:
\[
\frac{D^k}{k!}c_n(x)
=
\frac{1}{k!\,n!}D^ks_{n,0}(x+1)
=
\frac{1}{n!}s_{n,k}(x+1),
\]
where \eqref{DFsnkx} has been used. Evaluating at $x=0$ yields
\begin{equation}\label{cnk-stir}
c(n,k)
=
\frac{s_{n,k}(1)}{n!}.
\end{equation}

Finally, \eqref{IFHRRsnk-c} gives $s_{n,k}(1)=s(n-1,k)+s(n-1,k-1)$, so that \eqref{cnk-stir} becomes
\begin{equation}\label{F-cnk-1}
c(n,k)
=
\frac{s(n-1,k)+s(n-1,k-1)}{n!}.
\end{equation}

\begin{remark}
By the hockey-stick identity \cite[Theorem~E.5g]{Comt},
\begin{equation}\label{Bcoeff-Main}
\binom{x+1}{n}
=
\sum_{k=0}^{n}\binom{x-k}{\,n-k\,}.
\end{equation}
Hence the polynomial $c_n(x)=\binom{x+1}{n}$ admits, via \eqref{Bcoeff-Main}, an alternative
representation as a finite sum of shifted binomial polynomials. This gives a purely binomial
description of $c_n(x)$, complementary to the Stirling-number formula \eqref{F-cnk-1}.
\end{remark}
\section{Comtet numbers of two kinds, $b(n,k)$ and $B(n,k)$}\label{sec:comtet}
\subsection{More identities associated with $b(n,k)$}
We examine the Comtet numbers of the first kind, defined by \eqref{EGF-bnk}, and revisit the results  on $b(n,k)$ that Lehmer established, deriving additional properties.

First, we provide an alternative expression for 
$\s_n$ given by \eqref{def-sigma-n} in the following result.
\begin{theorem}
For $n\geq 2$,
\begin{equation}\label{sn-sum-formula}
\s_n = n! \sum_{k=1}^{[n/2]} c({n-k,k}).
\end{equation}
\end{theorem}
\begin{proof}
By substituting $t$ for $x$ in \eqref{OGF-cnx} and  equating this equation with \eqref{def-L-x2x}, we obtain 
$$ \sum_{j\geq0} c_j(x)x^j= \sum_{n\geq0} \sigma_n \frac{x^n}{n!}.$$
Substituting $c_j(x)= \sum_{k=0}^jc(j,k)x^k$ into the above equation and comparing the coefficients of $x^n$, we obtain 
$$ \frac{\sigma_n}{n!} =\sum_{j+k=n, j\geq k\geq1} c({j,k}),$$
leading to the formula given in \eqref{sn-sum-formula}.
\end{proof}
\begin{remark}
{\rm   If we substitute the expression for $c(n,k)$ in \eqref{F-cnk-1} into \eqref{sn-sum-formula}, we obtain 
the formula for $\s_n$ in \eqref{def-sigma-n}. 

Since $c(i,j)=0$ when $i<j$ and $c(j,0)=0$ for $j\geq2,$  the value of $\s_n$ can be determined as
    $$ \s_n = n! \sum_{k=0}^{n} c({n-k,k}).$$
    If we substitute the expression for $c(n,k)$ in \eqref{F-cnk-1} into this equation, we
    obtain 
    \begin{eqnarray}
       \s_n &=&\sum_{k=0}^{n} \frac{n!}{(n-k)!}\sum_{j=k}^n \binom{j}{k} s(n-k,j) \nonumber \\
       &=&\sum_{j=0}^{n} \sum_{k=0}^j k!\binom{n}{k} \binom{j}{k} s(n-k,j) \nonumber\\
       &=&\sum_{j=0}^{n}b(n,j), \label{Exp-sigma_n}
    \end{eqnarray}
    where use is made of the expression for $b(n,j)$ from \cite[Equation (12)]{Leh} in the last equation.
    }
\end{remark}
We establish the following important identity between $b(n,k)$ and $s_{n,k}(k)$.
\begin{theorem}\label{thm:bnk-snk}
    For $n\geq k\geq 0$,
    \begin{equation}\label{R-bnk-St1}
        b(n,k) = s_{n,k}(k).
    \end{equation}
\end{theorem}
\begin{proof}
The exponential generating functions for $b(n,k)$ and $s_{n,k}(k)$, given in \eqref{EGF-bnk} and \eqref{EGFsnkx} where $x=k$, respectively, coincide.
\end{proof}
Expanding $s_{n,k}(k)$ in Theorem~\ref{thm:bnk-snk} by means of \eqref{EEsnk1}, we obtain
\begin{equation}\label{exp-bnk-snj}
    b(n,k)= \sum_{j=k}^n\binom{j}{k}s(n,j)k^{j-k}.
\end{equation}
This gives a new proof of the explicit formula independently obtained by Comtet and Lehmer and leads to the following equivalent expressions for $b(n,k)$ in terms of the Stirling numbers $s(j,k)$.
\begin{proposition}
   \begin{eqnarray}
  b(n,k) &=& \sum_{j=k}^n\frac{n
  !}{j!} s(j,k)\binom{k}{n-j}=\sum_{j=k}^n\binom{n
  }{j} \frac{k!}{(k-n+j)!} s(j,k) .\label{EE-bnk-1}\\
   b(n,k) &=& (-1)^{n+k}\sum_{j=k}^n \binom{j}{k}s(n+1,j+1)(n-k)^{j-k}.\label{EE-bnk-3}\\
   b(n,k) &=& \sum_{j=k}^n\binom{j}{k}s(n+1,j+1)(k+1)^{j-k}.\label{EE-bnk-4}
   \end{eqnarray}
\end{proposition}
\begin{proof}
All the results are based on the key identity \eqref{R-bnk-St1}. Setting $y=k$ in \eqref{EEsnk2}
gives the second expression in \eqref{EE-bnk-1}, and the first follows from
\[
   \binom{n}{j}(k)_{n-j}
   =\frac{n!}{j!\,(n-j)!}\cdot\frac{k!}{(k-n+j)!}
   =\frac{n!}{j!}\binom{k}{\,n-j\,};
\]
setting $x=k$ in \eqref{EEsnk1} recovers Comtet's formula \eqref{exp-bnk-snj} instead. Then,
\eqref{EE-bnk-3} and \eqref{EE-bnk-4} are derived by expanding $ s_{n,k}(k)=r_{n,k}(k+1-n)$ in two different ways, as shown below:
\begin{eqnarray*}
    r_{n,k}(k+1-n) &=&\sum_{j=k}^n\binom{j}{k}r_{n,j}(1)(k-n)^{j-k}\\
    &=&\sum_{j=k}^n\binom{j}{k}r(n+1,j+1)(k-n)^{j-k} \\
    &=&(-1)^{n-k}\sum_{j=k}^n\binom{j}{k}s(n+1,j+1)(n-k)^{j-k}.
\end{eqnarray*}
and 
\begin{eqnarray*}
    r_{n,k}(k+1-n) &=&\sum_{j=k}^n\binom{j}{k}r_{n,j}(-n)(k+1)^{j-k}\\
    &=&\sum_{j=k}^n\binom{j}{k}s(n+1,j+1) (k+1)^{j-k}.
    \end{eqnarray*}
    %   For \eqref{EE-bnk-4}, we use the Taylor expansion for %$s_{n,k}(x)$ at $x=-1:$
    %$$ s_{n,k}(x) =\sum_{j=k}^n\binom{j}{k}s(n+1,j+1)(x+1)^{j-k}.
    %$$
    %The equation \eqref{EE-bnk-4} is immediate from setting $x=-k$ %in this equation. 
\end{proof}
%\begin{corollary}
%\begin{equation}
%\s_n = n! \sum_{k=1}^{[n/2]}
%\frac{s(n-k,k) + (n-k) s(n-k-1,k)}{(n-k)!}
%\end{equation}
%\end{corollary}
Substituting the second expression in  $\eqref{EE-bnk-1}$ into \eqref{Exp-sigma_n} and changing the order of summation yields 
\[
 \s_n =\sum_{j=0}^n \binom{n
  }{j}   \sum_{k=0}^j\frac{k!}{(k-n+j)!} s(j,k).
\]
Comtet's formula \eqref{exp-bnk-snj} can be compared with the following expression \eqref{bnk-cnj}.
\begin{proposition}
    \begin{eqnarray}
    b(n,k) &=& n!\sum_{j=k}^n\binom{j}{k}c(n,j)(k-1)^{j-k}.
    \label{bnk-cnj}\\
    c(n,k) &=& \sum_{j=k}^n\frac{1}{j!}\binom{1-k}{n-j}b(j,k).\label{bnk-cnj-b} 
    \end{eqnarray}
  %  \begin{eqnarray}
  %\frac{1}{n!}[b(n+1,j)+(n-k)b(n,k)] &=& \sum_{j=k}^n\binom{j}{k}c(n,j)k^{j-k}.
 %\end{eqnarray}
\end{proposition}
\begin{proof}
We use \eqref{R-bnk-St1} and \eqref{cnk-stir}  to derive 
 \eqref{bnk-cnj} by expanding $s_{n,k}(k)$ as follows:
 \begin{eqnarray*}
   b(n,k) &=&s_{n,k}(k)=s_{n,k}(1+k-1)=\sum_{j=k}^{n}\binom{j}{k}s_{n,j}(1)(k-1)^{j-k}\\
   &=& \sum_{j=k}^{n}\binom{j}{k}s_{n,j}(1)(k-1)^{j-k} =n!\sum_{j=k}^{n}\binom{j}{k}c({n,j})(k-1)^{j-k}.
 \end{eqnarray*}
 For \eqref{bnk-cnj-b}, we use \eqref{EEsnk2} to expand $c(n,k)$ as
  \begin{eqnarray*}
  c(n,k) &=&\frac{s_{n,k}(1)}{n!} =\frac{s_{n,k}(k+1-k)}{n!}\\
  &=& \frac{1}{n!}\sum_{j=k}^{n}\binom{n}{j}s_{j,k}(k)(1-k)_{n-j} 
   = \sum_{j=k}^{n}\frac{1}{j!}b(j,k)\binom{1-k}{n-j}.
\end{eqnarray*}
  
\end{proof}
The numbers $b(n,k)$ possess a convolution identity.
\begin{proposition}
   \begin{eqnarray*}
    \binom{k+j}{k} b(n,k+j) &=& \sum_{l=k}^{n-j}\binom{n}{l}
    b(l,k)b(n-l,j).
\end{eqnarray*}
\end{proposition}
\begin{proof}
    It follows by substituting $x=k$ and $y=j$ into \eqref{CFsnkx}. 
\end{proof}

Lehmer provided the following formula, as stated in \cite[Theorem 4]{Leh}.
    \begin{equation}\label{Leh-Thm4}
    \sum_{k=1}^n(-1)^k(x+k)^{k-1}b(n,k)=(-1)^n\binom{x+n-1}{n-1}(n-1)!,
\end{equation}

By utilizing the well-known identity  $(-x)^{\overline k} =(-1)^k(x)_k$, and performing an appropriate change of variables and indices, we derive from \eqref{Leh-Thm4} an expression for  $(x)_n$ in terms of $b(n,k),$ 
yielding:
\begin{equation}\label{Exp-ff-bnk}
    (x)_n = \sum_{j=0}^nb(n+1,j+1)(x-j)^{j}.
\end{equation}
Applying the differentiation formula $\frac{D^k}{k!}$ in  \eqref{DFsnkx} to this equation results in the expression:
\[
   s_{n,k}(x)= \sum_{j=k}^n \binom{j}{k}b(n+1,j+1)(x-j)^{j-k}.
\]
By means of \eqref{R-bnk-St1}, substituting $x=k$ into this equation gives the recurrence relation for $b(n,k)$ as stated in the following theorem.
\begin{theorem}
    $$ b(n,k)= \sum_{j=k}^n \binom{j}{k}b(n+1,j+1)(k-j)^{j-k}.$$  
\end{theorem}

A more general result than \eqref{Leh-Thm4} is given by \eqref{Thm4-gform}, from which we derive  an inversion formula for \eqref{exp-bnk-snj}.
\begin{theorem}
\begin{equation}\label{Thm4-gform}
 r_{n-1,\,j-1}(x+1) =\sum_{k=j}^n(-1)^{n+k}\binom{k-1}{j-1}b(n,k)(x+k)^{k-j}
 \qquad (j\ge1).
 \end{equation}
\begin{equation}\label{inv-exp-bnk-snj}
       s(n,j)= \sum_{k=j}^n(-1)^{k-j}\binom{k-1}{j-1}b(n,k)k^{k-j}.
\end{equation}
\end{theorem}
\begin{proof}
We begin by rewriting \eqref{Leh-Thm4} in a more compact form using the  falling factorial:
$$ (-1)^n r_{n-1,0}(x+1) =\sum_{k=1}^n(-1)^k(x+k)^{k-1}b(n,k).
    $$
    Next, we apply the differentiation formula $\frac{D_x^j}{j!}$ from \eqref{DFsnkx} to this equation, yielding:
\begin{equation}\label{Thm4-gform-pre}
     (-1)^n r_{n-1,j}(x+1) =\sum_{k=j+1}^n(-1)^k\binom{k-1}{j}(x+k)^{k-1-j}b(n,k).
\end{equation}
Replacing $j$ by $j-1$ throughout \eqref{Thm4-gform-pre}---on both sides---and multiplying by
$(-1)^n$ gives \eqref{Thm4-gform}.
Now, evaluating \eqref{Thm4-gform-pre} at $x=0$ and using the second equation in \eqref{my-rnkmn-EQ}, we
  get
  $$ 
  (-1)^n r(n,j+1) =\sum_{k=j+1}^n(-1)^k\binom{k-1}{j}k^{k-1-j}b(n,k),$$
 which transforms into \eqref{inv-exp-bnk-snj} after replacing $j$ by $j-1$ and applying  \eqref{R-rnksnk}. 
\end{proof}

It is of interest to give an alternative proof of Lehmer's result in \cite[Theorem 3]{Leh}.
\begin{proposition}\label{prop:sum-b-vanish}
   For $n>1,$
   $$ \sum_{k=1}^n (-1)^{k} (k-1)^{k-1}b(n,k)=0.$$
\end{proposition}
\begin{proof}
We first note the binomial-coefficient identity, valid for $j\geq1,$
\begin{equation}\label{ID-BC}
     \sum_{k=1}^j (-1)^{k}\binom{j}{k}(k-1)^{j-1} =(-1)^{j}.
\end{equation}
Indeed, since $(k-1)^{j-1}$ is a polynomial in $k$ of degree $j-1<j,$ its $j$th finite difference
vanishes,
\[
   \sum_{k=0}^j (-1)^{k}\binom{j}{k}(k-1)^{j-1}=0,
\]
and isolating the $k=0$ term, which equals $(-1)^{j-1},$ yields \eqref{ID-BC}. Equivalently, by means
of $\binom{j+1}{k+1}=\frac{j+1}{k+1}\binom{j}{k}$, the identity \eqref{ID-BC} may be recast as
\[
     \sum_{k=0}^j \frac{(-1)^{k}}{k+1}\binom{j}{k}k^{j} = \frac{(-1)^{j}}{j+1},
\]
which is a special case of Melzak's formula \cite[Equation~(7-1)]{QG} with $f(x)=x^{j}$ and
$(x,y)=(1,0).$
Now, using \eqref{bnk-cnj} and then \eqref{ID-BC}, we obtain, for $n>1,$
\begin{eqnarray*}
  \sum_{k=1}^n (-1)^{k} (k-1)^{k-1}b(n,k)
  &=&  n!\sum_{k=1}^n (-1)^{k} (k-1)^{k-1} \sum_{j=k}^n\binom{j}{k}c(n,j)(k-1)^{j-k}\\
  &=&  n!\sum_{j=1}^n c(n,j)\sum_{k=1}^j(-1)^{k} \binom{j}{k}(k-1)^{j-1}\\
  &=&  n!\sum_{j=1}^n (-1)^{j} c(n,j).
\end{eqnarray*}
By \eqref{cnx-Mono} the last sum equals $c_n(-1)-c(n,0)=\binom{0}{n}-\binom{1}{n},$ which vanishes for
$n>1.$ This proves the proposition.
\end{proof}

Using the reflection formula for $s_{n,k}(x)$, we derive vanishing properties of these polynomials, from which we obtain
the corresponding property for the Comtet numbers of the first kind (see~\cite[Theorem~7]{Leh}).

\begin{proposition}
\begin{enumerate}
\item[\textup{(i)}] For any integers $m$ and $h$ such that $0<h\leq 2m+1$,
\begin{equation}\label{vp-snkx}
s_{4m+1,2h}(2m)=0.
\end{equation}
\item[\textup{(ii)}] For any integers $m$ and $h$ such that $0<h\leq 2m+2$,
\begin{equation}\label{vp2-snkx}
s_{4m+3,2h}(2m+1)=0.
\end{equation}
\item[\textup{(iii)}] For any integer $h>0$,
\begin{equation}\label{vp-bnk}
b(4h+1,2h)=0.
\end{equation}
\end{enumerate}
\end{proposition}
\begin{proof}
Both \eqref{vp-snkx} and \eqref{vp2-snkx} follow by setting $x=0$ in \eqref{RF-Feq-snkx2}, since $n-k$ is odd. Finally, \eqref{vp-bnk} is the special case $m=h$ of \eqref{vp-snkx}, because $b(4h+1,2h)=s_{4h+1,2h}(2h)$ by \eqref{R-bnk-St1}.
\end{proof}
 \subsection{More identities associated with $B(n,k)$} \label{sec:Bnk}
Central to this subsection is the realization of the Comtet numbers of the second kind as the special
values $G_{n,k}(-n)=k!\,B(n+1,k+1)$ of the weighted Stirling polynomials. This realization yields a
variety of identities and combinatorial properties of $B(n,k)$, and provides simple alternative
proofs of several formulas due to Lehmer.

We first give an explicit formula for $B(n,k)$ in terms of the Stirling numbers of the second kind $S(j,k).$
 
\begin{theorem}
    \begin{eqnarray}
        B(n,k) &=& \sum_{j=k-1}^{n-1} \binom{n-1}{j}(1-n)^{n-1-j}S(j,k-1). \label{For-Bnk-S2} \\
    \end{eqnarray}
\end{theorem}
\begin{proof}
We begin with the polynomial identity from \cite[Equation (14)]{Leh}:
    \begin{equation}\label{Leh-poly}
       (-1)^n (x+n)^{n-1} = \sum_{k=1}^{n} (-1)^{k}(k-1)!B(n,k) \binom{x+k-1}{k-1}.
    \end{equation}
By replacing $x$ with $x-1$ in this equation, we obtain
$$  (-1)^n (x+n-1)^{n-1} =\sum_{k=0}^{n-1} (-1)^{k+1}k!B(n,k+1) \binom{x+k-1}{k}.$$
    Now, by the formula \eqref{CoS1kx4}, the coefficient of 
    $\binom{x+k-1}{k}$ on the right side of the above equation is found by 
    \begin{equation}\label{Int-Bnk}
        (-1)^{k+1}k!B(n,k+1) = \sum_{j=k}^{n-1}{a_j}K(j,k)=\sum_{j=k}^{n-1}(-1)^{j-k}{a_j}k!S(j,k),
    \end{equation}
    where $a_j=(-1)^n\binom{n-1}{j}(n-1)^{n-1-j}$ is the  coefficient of $x^j$ in the expansion 
    $$ (-1)^n (x+n-1)^{n-1} =(-1)^n \sum_{j=0}^{n-1}  \binom{n-1}{j}(n-1)^{n-1-j}x^j.$$
  By substituting the expression for $a_j$ into \eqref{Int-Bnk} and simplifying,   we arrive at the following expression:
   $$ 
   B(n,k+1) = \sum_{j=k}^{n-1} (-1)^{n-1-j} \binom{n-1}{j}(n-1)^{n-1-j}S(j,k). 
   $$
  Now, by making the substitution  $k\rightarrow k-1$, we obtain the desired result in \eqref{For-Bnk-S2}.
\end{proof}

\begin{corollary}
\begin{eqnarray*}
 B(n,1) &=& (-1)^{n-1}(n-1)^{n-1}, \\
 B(n,2)&=& (-1)^{n}\left((n-1)^{n-1}-(n-2)^{n-1}\right) \quad (n\geq 2),\\
 B(n,3)&=& \frac{(-1)^{n-1}}{2}\left((n-1)^{n-1}-2(n-2)^{n-1}+(n-3)^{n-1}\right) \quad (n\geq 3),\\
 B(n,n)&=&1,\\
 B(n,n-1)&=&-\binom{n}{2}\quad (n\geq2),\\
 B(n,n-2)&=&4\binom{n}{3}+3\binom{n}{4}\quad (n\geq2),\\
 B(n,n-3)&=&-\frac{(n+2)(n+5)}{2}\binom{n}{4}\quad (n\geq3).
\end{eqnarray*}
\end{corollary}
\begin{proof}
The identities for $B(n,1)$, $B(n,2)$, and $B(n,3)$ follow by setting $k=1,2,3$ in \eqref{For-Bnk-S2}
and using $S(j,1)=1$ and $S(j,2)=2^{j-1}-1$. The near-diagonal identities follow from
\eqref{For-Bnk-S2} together with the diagonal values $S(m,m)=1$, $S(m,m-1)=\binom{m}{2}$,
$S(m,m-2)=\binom{m}{3}+3\binom{m}{4}$, and $S(m,m-3)=\binom{m}{4}+10\binom{m}{5}+15\binom{m}{6}$,
for which we refer to \cite[Chapter~8]{Ch} and \cite[Chapter~6]{GKP}.
\end{proof}
 
The following identity realizes $B(n,k)$ as a special value of the weighted Stirling polynomial
$G_{n,k}(x)$; it is the second-kind counterpart of \eqref{R-bnk-St1}.
\begin{theorem}
For all $n\ge k\ge0$,
\begin{equation}\label{def-Gnk-Com2}
 G_{n,k}(-n) = k!\, B(n + 1, k + 1),
\end{equation}
or, equivalently, $G_{n-1,k-1}(1-n)=(k-1)!\,B(n,k)$ for $n\ge k\ge1$.
\end{theorem}
\begin{proof}
Up to the factor $(k-1)!$, the right-hand side of \eqref{For-Bnk-S2} is precisely the expansion
\eqref{defSP} of the weighted Stirling polynomial $G_{n-1,k-1}(x)$ evaluated at $x=1-n$. Thus
$B(n,k)$ is a special value of $G_{n-1,k-1}(x)$, namely $G_{n-1,k-1}(1-n)=(k-1)!\,B(n,k)$; replacing
$(n,k)$ by $(n+1,k+1)$ yields \eqref{def-Gnk-Com2} for $k\ge1$. The case $k=0$ holds as well: by
\eqref{defGnraa} one has $G_{n,0}(x)=x^{n}$, so the left-hand side is $(-n)^{n}$, while
$B(n+1,1)=(-1)^{n}n^{n}=(-n)^{n}$ by the first identity of the preceding corollary.
\end{proof}

Another expression for $B(n,k)$ is given by the following formula. 
\begin{theorem}
    \begin{equation}
        B(n,k)= \sum_{j=k}^{n} \binom{n-1}{j-1}(-n)^{n-j}S(j,k). \label{For-2-Bnk-S2}
    \end{equation}
\end{theorem}
\begin{proof}
We make use of \eqref{Leh-poly}.
  By  applying the formula from \eqref{CoS1kx3}, we can write the coefficient in \eqref{Leh-poly} as
    $$
   (-1)^{k+1}k!B(n,k+1) =\sum_{j=k}^{n-1}(-1)^n\binom{n-1}{j}n^{n-1-j}J(j,k).
    $$
 Simplifying this further, we get
    $$
   B(n,k+1) =\sum_{j=k}^{n-1}(-1)^{n-1-j}\binom{n-1}{j}n^{n-1-j}S(j+1,k+1).
    $$
Finally, shifting the indices from $(j,k)$ to $(j-1,k-1)$ yields the desired result \eqref{For-2-Bnk-S2}.  
\end{proof}
The following equation represents the inversion formula for \eqref{Exp-ff-bnk}:
\begin{equation}\label{Exp-power-ff}
    (x - n)^n = \sum_{k=0}^{n} B(n + 1, k + 1)\, (x)_{k}.
\end{equation}
This identity is obtained by substituting \( x = -n \) into \eqref{defGnraa}, 
applying \eqref{def-Gnk-Com2}, and then replacing \( y \) with \( x \).

The following result  presents  convolution identities for 
$B(n,k)$ and signed Stirling numbers of the first kind.
\begin{proposition}
\begin{eqnarray}
     \sum_{j=k}^nB(n+1,j+1)s(j,k) &=& \binom{n}{k}(-n)^{n-k}.\label{Pow-B-ff-a}\\
\sum_{j=k}^nB(n+1,j+1)s(j+1,k+1) &=& \binom{n}{k}(-n-1)^{n-k}.\label{Pow-B-ff-b}\\
\sum_{k=0}^nB(n+1,k+1)(n+r)_{k} &=& r^n . \label{Pow-B-ff-c}
\end{eqnarray}
\end{proposition}
\begin{proof}
Applying the differentiation formula $\frac{D^k}{k!}$ in \eqref{DFsnkx} to \eqref{Exp-power-ff} yields
 $$ \binom{n}{k}(x-n)^{n-k} =\sum_{j=k}^nB(n+1,j+1)s_{j,k}(x)$$
By setting $x=0$ and $x=-1$, we  derive \eqref {Pow-B-ff-a} and \eqref {Pow-B-ff-b}, using \eqref{IFHRRsnk} for the latter. Additionally, it is straightforward to see that
\eqref {Pow-B-ff-c} follows by substituting $x=n+r$ into \eqref{Exp-power-ff}.
\end{proof}

We present two expressions for $B(n,k)$ in terms of $S(n,j).$
\begin{proposition}
\begin{eqnarray}
  (-1)^{n-k}B(n+1,k+1) &=& \sum_{j=k}^n \frac{j!}{k!}\binom{n-k}{j-k}S(n,j).  \label{Rep-Bnkone} \\
  B(n+1,k+1) &=& \sum_{j=k}^n (-1)^{j-k}\frac{j!}{k!}\binom{n-1+j-k}{j-k}S(n,j). \label{Rep-Bnkone-b}
\end{eqnarray}
\end{proposition}
\begin{proof}
    We use the following polynomial identity from \eqref{defKnraa}:
\[
x^{n} = \sum_{j=0}^{n} (-1)^{\,n-j} j!\, S(n, j)\, \binom{x + j - 1}{j}.
\]
Applying the $k$th difference operator $\Delta ^k$  to this equation results in
$$ G_{n,k}(x) = \sum_{j=k}^{n}(-1)^{n-j}j!S(n,j)
    \binom{x+j-1}{j-k}. $$
    By setting $x=-n$ in this equation and using \eqref{def-Gnk-Com2}, we obtain
    $$
    k!B(n+1,k+1) = (-1)^{n-k}\sum_{j=k}^{n}j!S(n,j)
    \binom{n-k}{j-k}. 
    $$ 
    Hence this equation becomes \eqref{Rep-Bnkone}.
    Next, for \eqref{Rep-Bnkone-b}, we derive the following polynomial identity by applying \( \D^k \) to \eqref{defGnraa} with \( x = 0 \) and then replacing \( y \) with \( x \).
    $$ 
    G_{n,k}(x) = \sum_{j=0}^{n-k}(j+k)!S(n,j+k)\binom{x}{j}.
    $$
    By substituting $x=-n$ into this equation and using \eqref{def-Gnk-Com2}, we obtain the desired result in \eqref{Rep-Bnkone-b}.
\end{proof}

\begin{remark}
 {\rm  It is clear from \eqref{Rep-Bnkone} that the sign  of $B(n,k)$
 is determined by the condition
 $$ (-1)^{n-k}B(n,k) >0.$$}
\end{remark}
 
Lehmer \cite[Theorem 10]{Leh} established the following identity through a proof by induction on $n.$
\begin{proposition}
\begin{equation}\label{Exp-Bnk-power}
    (k-1)!B(n,k) =\sum_{\nu=0}^{k-1}(-1)^{n-k-\nu} \binom{k-1}{\nu}(n-1-\nu)^{n-1}.
\end{equation}
   \begin{proof}
       By \eqref{Exp-Gnkx}, replacing $(n,k)$ with $(n-1,k-1)$, setting $x=1-n$, and using
$(1-n+\nu)^{n-1}=(-1)^{n-1}(n-1-\nu)^{n-1}$, we obtain
$$
G_{n-1,k-1}(1-n)=\sum_{\nu=0}^{k-1}(-1)^{n-k-\nu}\binom{k-1}{\nu}
(n-1-\nu)^{n-1}.
$$
The result follows immediately from \eqref{def-Gnk-Com2}.
   \end{proof} 
\end{proposition}
\begin{remark}
{\rm     Using \cite[Theorem 9.1]{QG} and \eqref{Exp-power-ff}, we can  provide an alternative proof of \eqref{Exp-Bnk-power}. To do this, we rewrite \eqref{Exp-power-ff} as
     $$ (x-n)^n =\sum_{k=0}^n k!B(n+1,k+1)\binom{x}{k}.$$
     where the coefficient $ k!B(n+1,k+1)$ is given by 
     $$ k!B(n+1,k+1) =\Delta^k(x-n)^n|_{x=0} =\sum_{j=0}^{k}(-1)^{k-j}\binom{k}{j}(j-n)^{n}.$$
     By replacing $(n,k)$ with $(n-1,k-1),$ we obtain \eqref{Exp-Bnk-power}.}
\end{remark}
The convolution for $B(n,k)$ is provided in the following result, which is the exact counterpart of
the convolution for $b(n,k)$ obtained above.
\begin{proposition}\label{prop:convB}
    \begin{equation}\label{Conv-Bnk}
        \binom{k+j}{k}B(n,k+j) = \sum_{r=k}^{n-j} \binom{n}{r}B (r,k)\, B (n-r,j).
    \end{equation}
\end{proposition}
\begin{proof}
By the defining generating function \eqref{EGF-Bnk},
\[
\begin{aligned}
  \Bigl(\sum_{n}B(n,k)\tfrac{x^{n}}{n!}\Bigr)\Bigl(\sum_{n}B(n,j)\tfrac{x^{n}}{n!}\Bigr)
  &=\frac{\psi(x)^{k}}{k!}\cdot\frac{\psi(x)^{j}}{j!}
  =\binom{k+j}{k}\frac{\psi(x)^{k+j}}{(k+j)!}\\
  &=\binom{k+j}{k}\sum_{n}B(n,k+j)\frac{x^{n}}{n!},
\end{aligned}
\]
and the Cauchy product on the left gives \eqref{Conv-Bnk}.
\end{proof}

\begin{remark}
{\rm
The specialization $B(n+1,k+1)=G_{n,k}(-n)/k!$ of \eqref{def-Gnk-Com2} cannot be fed into the
convolution \eqref{Con-Gnkx} for $G_{n,k}(x)$, because the point of evaluation $-n$ depends on the
first index and therefore varies with the summation variable, whereas $x$ and $y$ in
\eqref{Con-Gnkx} are fixed. This is why \eqref{Conv-Bnk} is proved from the generating function
instead. For $b(n,k)$ no such obstruction arises, since there the point of evaluation is the second
index, which is constant along the convolution.
}
\end{remark}

We present a new summation formula for $B(n,k)$, the counterpart of the vanishing property for
$b(n,k)$ established in Proposition~\ref{prop:sum-b-vanish}.
\begin{proposition}
For $n>0$,
    \begin{equation}\label{sum-Bnk}
        \sum_{k=0}^n \frac{n!}{(n-k)!}B(n+1,k+1) =0.
    \end{equation}
\end{proposition}
\begin{proof}
Substituting $x=-n$ and $y=n$ into \eqref{defGnraa} gives
\[
\sum_{k=0}^n G_{n,k}(-n)\binom{n}{k}=(-n+n)^n=0\qquad(n>0).
\]
By \eqref{def-Gnk-Com2}, it is evident that  the
left-hand side equals $\sum_{k=0}^n \frac{n!}{(n-k)!}B(n+1,k+1)$, which proves \eqref{sum-Bnk}.
\end{proof}

We  provide several identities for sums associated with  $B(n,k)$, with the first identity \eqref{ev1-exp-HyH-power} given by Lehmer (see \cite[Theorem 6]{Leh}). 

 By substituting $-x$ for $x$ in  \eqref{Leh-poly}, we obtain
  \begin{equation}\label{def-HyH}
    (x+n)^n =\sum_{k=0}^n (-1)^{n-k} B(n+1,k+1)x^{\overline{k}}.
\end{equation} 
Using the fact \cite[Eqn. 48]{Cer2022} that 
\begin{equation}\label{HyH-der-Binom}
    H_{n}^{(x)}= \frac{d}{dx}\binom{x+n-1}{n} =\frac{1}{n!}\frac{d}{dx} x^{\overline{n}},  
\end{equation}
where $ H_{n}^{(x)}$ is the hyperharmonic polynomial (for generalized harmonic numbers and their
Riordan-array treatment, see \cite{CEl}), differentiating this equation yields
\begin{equation}\label{Power-HyH-exp}
    n(x+n)^{n-1} =\sum_{k=1}^n (-1)^{n-k} k! B(n+1,k+1)H_k^{(x)}.
\end{equation} 
\begin{proposition}
    \begin{eqnarray}
       \sum_{k=1}^{n}(-1)^{k}(k-1)!B(n+1,k+1)&=&(-1)^nn^n.\label{nn-Bnk}\\
       \sum_{k=2}^{n}(-1)^k(k-2)!B(n,k)&=&(-1)^n(n-1)^{n-1}.\label{ev1-exp-HyH-power}\\
      \sum_{k=2}^{n}(-1)^k(k-1)!B(n,k)H_{k-1}&=&(-1)^n (n-1)n^{n-2}. \label{ev2-exp-HyH-power}\\
      \sum_{k=1}^{n}(-1)^k(k-1)!B(n,k)&=&(-1)^nn^{n-1}.\label{ev3-exp-HyH-power}\\
      \sum_{k=2}^{n}(-1)^kk(k-2)!B(n,k)&=&(-1)^nn^{n-1}.\label{ev4-exp-HyH-power}
      \end{eqnarray}
      \begin{equation}
          \sum_{k=3}^{n}(-1)^k(k-3)!B(n,k) =(-1)^n\left[(n-1)^{n-1}-(n-2)^{n-1} - (n-1)(n-2)^{n-2}\right].\label{ev5-exp-HyH-power}
      \end{equation}
\end{proposition}
\begin{proof}
  From \eqref{HyH-der-Binom} one finds $H_k^{(0)}=\tfrac1k$ and $H_k^{(1)}=H_k$, the $k$th harmonic
  number. Setting $x=0$ in \eqref{Power-HyH-exp} gives \eqref{nn-Bnk}, and \eqref{ev1-exp-HyH-power}
  follows from it on replacing $n$ by $n-1$, shifting $k\mapsto k-1$, and multiplying by $(-1)^n$.
  Setting instead $x=1$ gives
  \[
     n(n+1)^{n-1}=\sum_{k=1}^{n}(-1)^{n-k}k!\,B(n+1,k+1)H_k,
  \]
  which yields \eqref{ev2-exp-HyH-power} after the same reindexing.  
   Note also from \eqref{HyH-der-Binom} that $\int_{0}^1 H_k^{(x)}dx =1$ for $k\geq1.$
  By  substituting $x$ with $x-1$ in  \eqref{Power-HyH-exp} 
  and integrating the resulting equation from 0 to $1$, we obtain 
  \begin{equation}\label{Middle-exp}
      \sum_{k=2}^{n}(-1)^k(k-1)!B(n,k)=(-1)^n(n^{n-1}-(n-1)^{n-1})
  \end{equation}
  Substituting the value $B(n,1)=(-1)^{n-1}(n-1)^{n-1}$ into this equation then gives us \eqref{ev3-exp-HyH-power}.
  Next, \eqref{ev4-exp-HyH-power} is obtained by adding \eqref{Middle-exp} and  \eqref{ev1-exp-HyH-power}. 
  Finally, to  verify \eqref{ev5-exp-HyH-power}, we evaluate the expression in \eqref{Power-HyH-exp} at $x=-1.$
  Since $H_{1}^{(-1)} =1$ and $H_{k}^{(-1)} =-\frac{1}{k(k-1)}$ for $k>1$ (see \cite{Cer2022}), so that
  $k!\,H_{k}^{(-1)} =-(k-2)!$ for $k\geq2,$ we obtain
  $$
    n(n-1)^{n-1} =(-1)^{n-1}B(n+1,2) - \sum_{k=2}^n (-1)^{n-k} (k-2)! B(n+1,k+1).
  $$
  Replacing $n$ by $n-1$ and then shifting the summation index $k$ to $k+1$, this becomes
 $$
    \sum_{k=3}^{n}(-1)^k(k-3)!B(n,k) = B(n,2) +(-1)^{n-1}(n-1)(n-2)^{n-2}.
  $$
  Finally, substituting the value of $B(n,2)$ recorded in the corollary following
\eqref{For-Bnk-S2} yields
  \eqref{ev5-exp-HyH-power}.
\end{proof}

Now, we give a combinatorial interpretation of \eqref{ev3-exp-HyH-power} in terms of the unsigned Comtet numbers of the second kind $T(n,k)$, defined by
\[
T(n,k)=(-1)^{\,n-k}B(n,k)
\]
as recorded in A354794 \cite{OEIS354794}. Then \eqref{ev3-exp-HyH-power} becomes
\[
\sum_{k=1}^{n}(k-1)!\,T(n,k)=n^{\,n-1}.
\]

The first interpretation follows naturally from \eqref{EGF-Bnk}:
\[
\sum_{n\ge k}T(n,k)\frac{z^n}{n!}
=\frac{\tilde\psi(z)^k}{k!},
\qquad
\tilde\psi(z)
=
\sum_{\nu\ge1}
(\nu-1)^{\nu-1}
\frac{z^\nu}{\nu!}.
\]
The coefficient $(\nu-1)^{\nu-1}$ counts the endofunctions on a
$(\nu-1)$-element labeled set (see, e.g.,
Flajolet and Sedgewick \cite[Chapter~II]{FlajoletSedgewick}).
Hence $\tilde\psi(z)$ is the exponential generating function for
endofunction blocks.

Moreover, $(k-1)!$ is the number of cyclic orderings of $k$ labeled
objects, so weighting by $(k-1)!$ corresponds to the classical cyclic
construction in the symbolic method
\cite[Chapter~I]{FlajoletSedgewick}. Consequently,
\[
\sum_{k\ge1}\frac{\tilde\psi(z)^k}{k}
=
-\log\!\left(1-\tilde\psi(z)\right).
\]
Since
\[
1-\tilde\psi(z)=e^{-\mathcal T(z)},
\]
we obtain
\[
-\log\!\left(1-\tilde\psi(z)\right)
=
\mathcal T(z),
\]
where
\[
\mathcal T(z)
=
\sum_{n\ge1}
n^{\,n-1}\frac{z^n}{n!}
=
-W(-z)
\]
is the classical tree function satisfying
$\mathcal T=z\,e^{\mathcal T}$
(\cite{Cayley1889,Lambert1758,Corless1996}).

Therefore the weighted sum
\[
\sum_{k=1}^{n}(k-1)!\,T(n,k)
\]
counts cyclic assemblies of endofunction blocks, and the above identity
asserts that these objects are equinumerous with rooted labeled trees
on $n$ vertices, whose number is $n^{\,n-1}$ by Cayley's theorem
\cite{Cayley1889}.

A second interpretation of $T(n,k)$ is provided by the following set-partition formula, recorded
in A354794 \cite{OEIS354794} (N.~Skirrow). We prove it algebraically by means of the reflection formula.
\begin{proposition}
For all $n\ge k\ge1$, the unsigned Comtet number $T(n,k)$ is an $r$-Stirling number of the second
kind in the sense of Broder \cite{Bro}, with $r=n-k$:
\begin{equation}\label{T-rStirling}
T(n,k)=\bracenom{2n-k-1}{n-1}_{n-k},
\qquad\text{equivalently}\qquad
B(n,k)=(-1)^{\,n-k}\bracenom{2n-k-1}{n-1}_{n-k}.
\end{equation}
\end{proposition}
\begin{proof}
Writing $\bracenom{2n-k-1}{n-1}_{n-k}=\bracenom{(n-1)+r}{(k-1)+r}_{r}$ with
$r=n-k$, the relation \eqref{Broder-Gnk} between the $r$-Stirling numbers and $G_{n,k}(x)$ gives
\[
\bracenom{2n-k-1}{n-1}_{n-k}=\frac{1}{(k-1)!}\,G_{n-1,k-1}(n-k).
\]
The reflection formula \eqref{RecF-Gnkx} yields $G_{n-1,k-1}(n-k)=(-1)^{\,n+k}G_{n-1,k-1}(1-n)$, and
\eqref{def-Gnk-Com2} gives $G_{n-1,k-1}(1-n)=(k-1)!\,B(n,k)$. Hence
\[
\bracenom{2n-k-1}{n-1}_{n-k}=(-1)^{\,n-k}B(n,k)=T(n,k),
\]
which proves \eqref{T-rStirling}.
\end{proof}
Combinatorially, \eqref{T-rStirling} states that $T(n,k)$ counts the set partitions of
$\{1,\dots,2n-k-1\}$ into $n-1$ blocks in which the first $n-k$ elements lie in distinct blocks. For
instance, $T(4,2)=\bracenom{5}{3}_{2}=19$ and $T(5,2)=\bracenom{7}{4}_{3}=175$, so
that $B(4,2)=19$ and $B(5,2)=-175$.

\subsection{Recurrence relations for $b$ and $B$}\label{sec:rec}
In this subsection we establish recurrence relations for both Comtet families, working from the
properties of the weighted Stirling polynomials $s_{n,k}(x)$ and $G_{n,k}(x)$. For $B(n,k)$ we first
prove the main result, Theorem~\ref{thmA} of the introduction, which answers Lehmer's question by
providing a recurrence with a fixed number of terms. We then give a new proof of Comtet's four-term
recurrence \eqref{RR-b-intro} for $b(n,k)$.

We now prove the main result, Theorem~\ref{thmA}.

\begin{theorem}\label{thm-main-B}
  \begin{equation}\label{RR-B}
  kB(n+1,k+1)=B(n,k-1)-(n-k)B(n,k)-B(n+1,k).
\end{equation}
\end{theorem}
\begin{proof}
From the recurrence \eqref{RR-Gnkx} and the specialization
\eqref{def-Gnk-Com2}, with $x=-n$, we obtain
\[
G_{n+1,k}(-n)
=
-(n-k)\,k!\,B(n+1,k+1)
+
k!\,B(n+1,k).
\]

On the other hand, the difference relation
\[
G_{n+1,k}(x+1)
=
G_{n+1,k}(x)
+
G_{n+1,k+1}(x),
\]
which follows from the first identity in \eqref{D-Gnr}, gives, upon setting
$x=-n-1$ and using \eqref{def-Gnk-Com2},
\[
G_{n+1,k}(-n)
=
k!\,B(n+2,k+1)
+
(k+1)!\,B(n+2,k+2).
\]

Equating the two expressions for $G_{n+1,k}(-n)$ and dividing by $k!$ yields
\[
B(n+2,k+1)
+
(k+1)B(n+2,k+2)
=
-(n-k)B(n+1,k+1)
+
B(n+1,k).
\]

Replacing $(n,k)$ by $(n-1,k-1)$ gives
\[
kB(n+1,k+1)
=
B(n,k-1)
-
(n-k)B(n,k)
-
B(n+1,k),
\]
which is \eqref{RR-B}.
\end{proof}

In the unsigned normalization $T(n,k)=(-1)^{\,n-k}B(n,k)$, Theorem~\ref{thm-main-B} takes the
nonnegative form \eqref{RR-T-intro}, the recurrence conjectured by Kurkov in A354794 \cite{OEIS354794}.

We next treat $b(n,k)$.
By setting $x=k$ in the recurrence relation \eqref{RRsnkx} for $s_{n,k}(x)$ and using \eqref{R-bnk-St1}, we get
\begin{eqnarray}
    b(n+1,k) &=&s_{n+1,k}(k) = (k-n)s_{n,k}(k)+ s_{n,k-1}(k) \nonumber \\
    &=& (k-n) b(n,k) +s_{n,k-1}(k). \label{RR-b-exp1}
\end{eqnarray}
Next, by substituting $x=1$ and $y=k$ into the addition formula \eqref{AFsnkx2} and using \eqref{R-bnk-St1} again, we have
\begin{eqnarray}
s_{n,k}(k+1) &=& s_{n,k}(k)+ ns_{n-1,k}(k) \nonumber  \\
&=&b(n,k) +nb(n-1,k).\label{RR-b-exp2}
\end{eqnarray}
By combining \eqref{RR-b-exp1} with \eqref{RR-b-exp2} after replacing $k$ by $k-1$ in the latter, we obtain Comtet's four-term recurrence
\[
    b(n+1,k)=nb(n-1,k-1)+b(n,k-1)-(n-k)b(n,k),
\]
which is \eqref{RR-b-intro}.

%\begin{remark}
 % {\rm  The well known recurrence relation for $S(n,k)$ can be obtained by setting $x=0$ in \eqref{RR-Gnkx} and dividing the resulting equation by $k!$:
  %  $$ S({n+1,k})= kS({n,k}) +  S({n,k-1}).$$}
%\end{remark}

\subsection{Touchard-type polynomials and partial Bell polynomial representations}\label{sec:bell}
It is interesting to study the Touchard-type polynomials $\sigma_n(t)$ and $\Sigma_n(t)$ attached to the
two Comtet families. Starting from their differentiation formulas and exponential generating functions, we
derive recurrence and differential--difference relations for these polynomials and, finally,
recover the known realizations of $b(n,k)$ and $B(n,k)$ as special values of the partial Bell
polynomials.

The Bell number, denoted $B_n,$ is defined by the row sums for $S=\{S(n,k)\}:$
$$ B_n =\sum_{k=0}^n S(n,k)$$
Similarly,  as is defined in \cite{Leh}, $\Sigma_n$ denotes the row sums for  $B=\{B(n,k)\}.$ 
That is,
$$ \Sigma_n =\sum_{k=0}^n B(n,k).$$
In terms of the Bell numbers, two explicit expressions for $\Sigma_n$ follow immediately from \eqref{For-Bnk-S2} and  \eqref{For-2-Bnk-S2}.
\begin{proposition}
\begin{eqnarray}
    \Sigma_n &=& \sum_{k=0}^{n-1} \binom{n-1}{k}(1-n)^{n-1-k} B_k. \label{Sigma is a sum} \\
    \Sigma_n &=& \sum_{k=0}^{n} \binom{n-1}{k-1}(-n)^{n-k} B_k.
\end{eqnarray}
\end{proposition}

We define the Touchard-type polynomial for $b(n,k)$ by
$$ \sum_{k=0}^n b(n,k)t^k =\sigma_n(t),$$
and, similarly, $\Sigma_n(t)$ denotes the counterpart for $B(n,k),$ so that
$$ \sum_{k=0}^n B(n,k)t^k=\Sigma_n(t).$$
By the inversion (orthogonality) relation of Lehmer \cite{Leh}, the preceding two equations yield
\begin{equation}\label{Mono-bnk-Bnk}
\sum_{k=0}^n B(n,k)\sigma_k(t) =t^n; \quad
 \sum_{k=0}^n b(n,k)\Sigma_k(t) =t^n.
\end{equation}
In particular, the two families $\sigma_k(t)$ and $\Sigma_k(t)$ each form a basis of ${\mathcal{P}}_n(t).$
\begin{proposition}
For any $f(t)=\sum_{j=0}^n a_jt^j \in {\mathcal{P}}_n(t),$
\begin{eqnarray}
      f(t) &=&\sum_{k=0}^{n} \left(\sum_{j=k}^{n}
      a_j B(j,k) \right)\sigma_k(t); \label{Coskx} \\
      f(t) &=&\sum_{k=0}^{n} \left(\sum_{j=k}^{n}
      a_j b(j,k) \right)\Sigma_k(t).  \label{CoSkx}
 \end{eqnarray}
\end{proposition}
\begin{proof}
  It follows immediately from the two expressions in \eqref{Mono-bnk-Bnk}.
\end{proof}
Using $'$ to denote differentiation, we next present recurrence relations for $\sigma_n(t)$ and
$\Sigma_n(t)$ expressed through their derivatives.
\begin{proposition}
\begin{eqnarray}
\sigma_{n+1}(t) &=& nt\sigma_{n-1}(t) +(t-n)\sigma_{n}(t) +t\sigma_{n}'(t); \label{EXp-deri-RR} \\
     \Sigma_{n+1}'(t) &=& (t-n)\Sigma_{n}(t) +t\Sigma_{n}'(t) +(t^{-1}-1)\Sigma_{n+1}(t),
     \label{EXp-deri-RR-b}
\end{eqnarray}
where $t^{-1}\Sigma_{n+1}(t)$ is again a polynomial because $B(n+1,0)=0$.
\end{proposition}
\begin{proof}
Multiplying \eqref{RR-b-intro} by $t^{k}$ and summing over $k$, the four terms become
\[
  \sigma_{n+1}(t),\quad nt\,\sigma_{n-1}(t),\quad t\,\sigma_{n}(t),\quad -n\sigma_n(t)+t\sigma_n'(t),
\]
where we used $\sum_{k}kb(n,k)t^{k}=t\sigma_n'(t)$. This is \eqref{EXp-deri-RR}. Doing the same with \eqref{RR-B}, the left-hand side gives
\[
\begin{aligned}
  \sum_{k\ge0}k\,B(n+1,k+1)t^{k}
  &=\sum_{k\ge0}(k+1)B(n+1,k+1)t^{k}-\sum_{k\ge0}B(n+1,k+1)t^{k}\\
  &=\Sigma_{n+1}'(t)-t^{-1}\Sigma_{n+1}(t),
\end{aligned}
\]
while the right-hand side gives
$t\,\Sigma_n(t)-n\Sigma_n(t)+t\Sigma_n'(t)-\Sigma_{n+1}(t)$; rearranging yields
\eqref{EXp-deri-RR-b}.
\end{proof}
\begin{remark}
{\rm By setting $t=1$ in \eqref{EXp-deri-RR-b} and using \eqref{Sigma is a sum}, we get
$$ \Sigma_{n+1}'(1) -\Sigma_{n}'(1) = (1-n)\Sigma_{n} = \sum_{k=0}^{n-1} \binom{n-1}{k}(1-n)^{n-k} B_k,$$
where $\Sigma_{n}$ denotes $\Sigma_{n}(1).$}
\end{remark}
Lehmer showed in \cite[Theorem 1]{Leh} that the exponential generating function of $\s_n(t)$ is
\begin{equation}\label{def-snkt-EGF}
    \sum_{n\geq0}\s_n(t) \frac{x^n}{n!} =(1+x)^{t(1+x)}.
\end{equation}
From this we derive a differential--difference equation for $\s_n(t).$
\begin{proposition}
\begin{equation}\label{EXp-Deri-sigmax}
   \s_n'(t) = n\s_{n-1}(t) + n!\sum_{k=1}^{n-1}
   \frac{(-1)^{k-1}}{k(k+1)}\frac{\s_{n-1-k}(t)}{(n-1-k)!}.
\end{equation}
\end{proposition}
\begin{proof}
    By differentiating \eqref{def-snkt-EGF} with respect to $t$, we get
    $$
    \sum_{n\geq0} \s_n'(t) \frac{x^n}{n!} =D_t (1+x)^{t(1+x)} = (1+x)^{t(1+x)} (1+x)\log (1+x).
    $$
  Using the Taylor expansion of $\log (1+x)$, we find the Taylor expansion of $(1+x)\log (1+x):$
  \begin{equation} \label{Taylor-xlogx}
    (1+x)\log (1+x) = x +\sum_{k\geq 1} \frac{(-1)^{k-1}}{k(k+1)}x^{k+1}.
  \end{equation}
 Combining the two equations above we obtain
 \begin{eqnarray*}
   \sum_{n\geq0} \s_n'(t) \frac{x^n}{n!} &=& \sum_{j\geq0}\s_j(t) \frac{x^j}{j!}
    \left( x +\sum_{k\geq 1} \frac{(-1)^{k-1}}{k(k+1)}x^{k+1}  \right) \\
    &=& \sum_{j\geq0}\s_j(t) \frac{x^{j+1}}{j!}
    + \sum_{j\geq0}\s_j(t) \frac{x^j}{j!} \sum_{k\geq 1} \frac{(-1)^{k-1}}{k(k+1)}x^{k+1}.
 \end{eqnarray*}
 Next, by applying the Cauchy product to the right-hand side and equating the coefficients of $x^n$,
 we obtain the explicit expression for $\s_n'(t)$ given in \eqref{EXp-Deri-sigmax}.
\end{proof}
We present a recurrence relation for $\s_n:=\s_n(1).$
\begin{theorem}
\begin{equation}\label{RR-sigma}
 \s_{n+1} =(1-n)\s_n +2n\s_{n-1} +
n!\sum_{k=1}^{n-1}
   \frac{(-1)^{k-1}}{k(k+1)}\frac{\s_{n-1-k}}{(n-1-k)!}.
\end{equation}
\end{theorem}
\begin{proof}
    By equating the two expressions for $\s_n'(t)$ from \eqref{EXp-deri-RR} and \eqref{EXp-Deri-sigmax}
    and evaluating at $t=1$, we obtain \eqref{RR-sigma}.
\end{proof}

We now present alternative forms of the exponential generating functions of $\sigma_n(t)$ and
$\Sigma_n(t)$. Using \eqref{def-snkt-EGF} and \eqref{Taylor-xlogx}, we obtain
\begin{eqnarray}
   \sum_{n\geq0}\sigma_n(t) \frac{x^n}{n!} &=&
   (1+x)^{t(1+x)}= \exp({t(1+x)\log(1+x)}) \nonumber\\ 
   &=& \exp(tx +t\sum_{\nu\geq 1} \frac{(-1)^{\nu-1}}{\nu(\nu+1)}x^{\nu+1}). \label{Exp-EGF-snx}
\end{eqnarray}
Similarly, using \cite[Equations (5) and (6)]{Leh}, we obtain 
\begin{eqnarray}
   \sum_{n\geq0}\Sigma_n(t) \frac{x^n}{n!} &=&
    \sum_{n\geq0}\sum_{k=0}^{n} B(n,k)t^k \frac{x^n}{n!} =\sum_{k\geq0}t^k  \sum_{n\geq0} B(n,k)\frac{x^n}{n!}\nonumber  \\
 &=&
    \sum_{k\geq0}t^k \frac{{\psi(x)}^k}{k!}
    = \exp(t \psi(x)) \nonumber\\
    &=& \exp(t \sum_{\nu\geq1}(-1)^{\nu-1}(\nu-1)^{\nu-1}\frac{x^\nu}{\nu!}). \label{EXp-EGF-Snx}
\end{eqnarray}

By \cite[Equation 3b on p.134]{Comt}, a series $\exp(\sum_{m\geq 1}t_m\frac{x^m}{m!})$ in countably
many variables $t_i$ is expressed through the complete Bell polynomials ${\bf Y}_{n}(t_1,\cdots,t_n)$ as
$$ \exp(\sum_{m\geq 1}t_m\frac{x^m}{m!})
=1 + \sum_{n\geq 1}{\bf Y}_{n}(t_1,\cdots,t_n)\frac{x^n}{n!},$$
with
$${\bf Y}_{n}(t_1,\cdots t_n) = \sum_{k=1}^n {\bf B}_{n,k}(t_1, \cdots, t_{n-k+1}).$$
%where for $1\leq \nu \leq n-k+1,$
%$$ t_\nu = (-1)^{\nu}(\nu-1)^{\nu-1}\frac{1}%{\nu!}t.$$

The partial Bell polynomials ${\bf B}_{n,k}$ are explicitly given by 
\[
\begin{aligned}
{\bf B}_{n,k}(t_1, \cdots, t_{n-k+1})
&=\sum \frac{n!}{c_1!\,c_2!\cdots c_{n-k+1}!\;(1!)^{c_1}(2!)^{c_2}\cdots\bigl((n-k+1)!\bigr)^{c_{n-k+1}}}\\
&\quad\times t_1^{c_1}t_2^{c_2} \cdots t_{n-k+1}^{c_{n-k+1}},
\end{aligned}
\]
where the sum runs over all non-negative integers $c_1,c_2,\dots,c_{n-k+1}$ such that
$$ c_1+2c_2+ \cdots +(n-k+1)c_{n-k+1}=n ;\quad c_1+c_2+\cdots +c_{n-k+1}=k.$$
For $1\leq \nu \leq n-k+1,$ by choosing either $$
t_\nu =(-1)^{\nu}(\nu-2)!\,t \quad (\nu\geq 2);\quad t_1=t$$
or
$$ t_\nu = (-1)^{\nu-1}(\nu-1)^{\nu-1}t$$ from \eqref{Exp-EGF-snx} and \eqref{EXp-EGF-Snx} respectively,
we recover the following representations.
\begin{proposition}\label{thm-bell}
\begin{eqnarray*}
{\bf B}_{n,k}(1,\, 1,\, -1,\, 2,\, \cdots,\, (-1)^{n-k+1}(n-k-1)!)  &=&b(n,k); \\
   {\bf B}_{n,k}(1,\, -1,\, 4,\, -27,\, \cdots,\, (-1)^{n-k}(n-k)^{n-k})  &=&B(n,k),
\end{eqnarray*}
\end{proposition}
\begin{remark}
\normalfont
Proposition~\ref{thm-bell} is included for completeness; in each identity the argument list contains
$n-k+1$ entries and reduces to the single value $1$ when $k=n$. It is Comtet's fundamental formula
\cite[Th\'eor\`eme A, (3a), p.~134]{Comt} applied to \eqref{EGF-bnk} and \eqref{EGF-Bnk}. Both
evaluations are recorded in the \emph{On-Line Encyclopedia of Integer Sequences} as the Bell
transforms of $(-1)^{n-1}(n-1)!$ and of $(-n)^n$ respectively (P.~Luschny, 16 January 2016, in
A008296 \cite{OEIS8296} and A039621 \cite{OEIS39621}); the unsigned array A354794
\cite{OEIS354794} is defined there as the Bell transform of $\{m^m\}_{m\ge0}$.
\end{remark}

% =====================================================================
\appendix
\setcounter{equation}{0}

\section{Proofs of the identities of Section~\ref{sec:wsp}}\label{app:proofs}

The proofs of the identities in Theorems~\ref{BPsnkx} and~\ref{BPGnkx} and in
Proposition~\ref{prop:GJKH} are collected here to keep the paper self-contained. They rest on the
two expansions
\begin{equation}\label{A-mother}
   \sum_{n\ge0}(x+y)_n\frac{t^n}{n!}=(1+t)^{x+y},
   \qquad
   \sum_{n\ge0}(x+y)^n\frac{t^n}{n!}=e^{(x+y)t},
\end{equation}
the first being Newton's binomial series and the second the exponential series, together with the
Vandermonde convolution \cite[Chapter~5]{GKP}: for indeterminates $x,y$ and integers $n,j\ge0$,
\begin{equation}\label{A-vdm-ff}
   (x+y)_n=\sum_{i=0}^{n}\binom{n}{i}(x)_{n-i}(y)_i,
\end{equation}
\begin{equation}\label{A-vdm-bin}
   \binom{x+y}{j}=\sum_{i=0}^{j}\binom{x}{\,j-i\,}\binom{y}{i}.
\end{equation}
Throughout we use that $\{y^k\}_{k\ge0}$ and $\bigl\{\binom{y}{k}\bigr\}_{k\ge0}$ are each a basis of
$\mathbb{C}[y]$, so that comparing coefficients with respect to either family is legitimate.

% ---------------------------------------------------------------------
\subsection{The first kind}\label{app:first}

\begin{proof}[Proof of Theorem~\ref{BPsnkx}]
We begin the proof by recalling the defining expansion \eqref{defsnkx}, $(x+y)_n=\sum_{k}s_{n,k}(x)y^{k}$, together with the equalities above.

\smallskip
\noindent\emph{(v) Double generating function.} Substituting \eqref{defsnkx} into the first identity
of \eqref{A-mother} gives \eqref{DGFsnkx} at once.

\smallskip
\noindent\emph{(iv) Exponential generating function.} Writing
\[
  (1+t)^{x+y}=(1+t)^{x}e^{\,y\ln(1+t)}=(1+t)^{x}\sum_{k\ge0}\frac{\bigl(\ln(1+t)\bigr)^{k}}{k!}\,y^{k}
\]
and comparing the coefficient of $y^{k}$ in \eqref{DGFsnkx} gives \eqref{EGFsnkx}.

\smallskip
\noindent\emph{(iii) Recurrences.} From $(x+y)_{n+1}=(x+y-n)(x+y)_n$,
\[
  \sum_{k}s_{n+1,k}(x)y^{k}
  =(x-n)\sum_{k}s_{n,k}(x)y^{k}+\sum_{k}s_{n,k}(x)y^{k+1};
\]
comparing the coefficients of $y^{k}$ gives \eqref{RRsnkx}, and setting $x=0$ there gives
\eqref{RRsnk}.

\smallskip
\noindent\emph{(i) Addition formulas.} Apply \eqref{defsnkx} first with the pair $(x+y,\,z)$ and
then with the pair $(y,\,x+z)$:
\[
\begin{aligned}
  \sum_{k}s_{n,k}(x+y)z^{k}
  &=\bigl((x+y)+z\bigr)_n
  =\bigl(y+(x+z)\bigr)_n
  =\sum_{j}s_{n,j}(y)(x+z)^{j}\\
  &=\sum_{j}s_{n,j}(y)\sum_{k}\binom{j}{k}x^{\,j-k}z^{k}.
\end{aligned}
\]
Comparing the coefficients of $z^{k}$ gives \eqref{AFsnkx}. For \eqref{AFsnkx2}, apply
\eqref{A-vdm-ff} to the pair $(x,\,y+z)$ and then expand $(y+z)_j$ by \eqref{defsnkx}:
\[
  \sum_{k}s_{n,k}(x+y)z^{k}
  =\bigl(x+(y+z)\bigr)_n
  =\sum_{j}\binom{n}{j}(x)_{n-j}(y+z)_j
  =\sum_{j}\binom{n}{j}(x)_{n-j}\sum_{k}s_{j,k}(y)z^{k},
\]
and comparing the coefficients of $z^{k}$ gives \eqref{AFsnkx2}.

\smallskip
\noindent\emph{(ii) Explicit expressions.} Setting $y=0$ in \eqref{AFsnkx} and in \eqref{AFsnkx2},
and using $s_{n,k}(0)=s(n,k)$, gives \eqref{EEsnk1} and \eqref{EEsnk2} respectively.

\smallskip
\noindent\emph{(vi) Convolution.} By \eqref{EGFsnkx}, with $L=\ln(1+t)$,
\[
  \Bigl(\sum_{n}s_{n,k}(x)\tfrac{t^{n}}{n!}\Bigr)
  \Bigl(\sum_{n}s_{n,j}(y)\tfrac{t^{n}}{n!}\Bigr)
  =(1+t)^{x+y}\frac{L^{k+j}}{k!\,j!}
  =\binom{k+j}{k}\sum_{n}s_{n,k+j}(x+y)\frac{t^{n}}{n!},
\]
and the Cauchy product on the left gives \eqref{CFsnkx}.
\end{proof}

% ---------------------------------------------------------------------
\subsection{The second kind}\label{app:second}

In this subsection we establish Proposition~\ref{prop:KJH-G}, Theorem~\ref{BPGnkx}, and
Proposition~\ref{prop:GJKH} in this order, as the proof of Proposition~\ref{prop:GJKH} relies on both
the substitutions given in Proposition~\ref{prop:KJH-G} and the identities established in
Theorem~\ref{BPGnkx}. We therefore begin with the substitutions expressing $J$, $K$, and $H$ in terms
of $G$.

\begin{proof}[Proof of Proposition~\ref{prop:KJH-G}]
Since
$\binom{-y+k-1}{k}=\frac{(-y+k-1)(-y+k-2)\cdots(-y)}{k!}=(-1)^{k}\binom{y}{k}$,
replacing $(x,y)$ by $(-x,-y)$ in \eqref{defKnraa} gives
\[
  (-1)^{n}(x+y)^{n}=\sum_{k}(-1)^{k}K_{n,k}(-x)\binom{y}{k},
  \qquad\text{i.e.}\qquad
  (x+y)^{n}=\sum_{k}(-1)^{\,n+k}K_{n,k}(-x)\binom{y}{k};
\]
comparison with \eqref{defGnraa} gives \eqref{R-KnrGnr}. Next, $\binom{y+k}{k}=\binom{(y+1)+k-1}{k}$,
so applying \eqref{defKnraa} to the pair $(x-1,\,y+1)$ yields
$(x+y)^{n}=\sum_{k}K_{n,k}(x-1)\binom{y+k}{k}$; comparison with \eqref{defJnraa} gives
$J_{n,k}(x)=K_{n,k}(x-1)$, which combined with \eqref{R-KnrGnr} is \eqref{R-JnrGnr}. Finally,
applying \eqref{defGnraa} to the pair $(x+1,\,y-1)$ yields
$(x+y)^{n}=\sum_{k}G_{n,k}(x+1)\binom{y-1}{k}$, and comparison with \eqref{defHnraa} gives
\eqref{R-HnrGnr}.
\end{proof}

\begin{proof}[Proof of Theorem~\ref{BPGnkx}]
Recall the defining expansion \eqref{defGnraa}, $(x+y)^{n}=\sum_{k}G_{n,k}(x)\binom{y}{k}$.

\smallskip
\noindent\emph{(v) Double generating function.} Substituting \eqref{defGnraa} into the second
identity of \eqref{A-mother} gives \eqref{DGFGnkx}.

\smallskip
\noindent\emph{(iv) Exponential generating function.} Since $e^{t}-1$ has zero constant term,
Newton's binomial series applies to it, and
\[
  e^{(x+y)t}=e^{xt}\bigl(1+(e^{t}-1)\bigr)^{y}
            =e^{xt}\sum_{k\ge0}\binom{y}{k}\bigl(e^{t}-1\bigr)^{k}.
\]
Comparing the coefficients of $\binom{y}{k}$ in \eqref{DGFGnkx} gives \eqref{stipol}.

\smallskip
\noindent\emph{(ii) Explicit expressions.} Expanding
$(e^{t}-1)^{k}=\sum_{\nu=0}^{k}(-1)^{\,k-\nu}\binom{k}{\nu}e^{\nu t}$ in \eqref{stipol} gives
$e^{xt}(e^{t}-1)^{k}=\sum_{\nu}(-1)^{\,k-\nu}\binom{k}{\nu}e^{(x+\nu)t}$, and the coefficient of
$t^{n}/n!$ is \eqref{Exp-Gnkx}. Taking instead the Cauchy product of $e^{xt}$ with
$\sum_{n}G(n,k)t^{n}/n!=(e^{t}-1)^{k}$, which is \eqref{stipol} at $x=0$, gives \eqref{defSP}.

\smallskip
\noindent\emph{(iii) Recurrences.} Differentiating \eqref{stipol} with respect to $t$ and writing
$e^{t}=(e^{t}-1)+1$ in the second term,
\[
  \partial_t\Bigl[e^{xt}\bigl(e^{t}-1\bigr)^{k}\Bigr]
  =x\,e^{xt}\bigl(e^{t}-1\bigr)^{k}+k\,e^{xt}\bigl(e^{t}-1\bigr)^{k-1}e^{t}
  =(x+k)\,e^{xt}\bigl(e^{t}-1\bigr)^{k}+k\,e^{xt}\bigl(e^{t}-1\bigr)^{k-1}.
\]
Comparing the coefficients of $t^{n}/n!$ gives \eqref{RR-Gnkx}, and setting $x=0$ there gives
\eqref{RR-Gnk}.

\smallskip
\noindent\emph{(i) Addition formulas.} Identity \eqref{AFGnkx} is the Cauchy product of $e^{xt}$
with $e^{yt}(e^{t}-1)^{k}$. For \eqref{AFGnkx2}, apply \eqref{defGnraa} first with the pair
$(x+y,\,z)$ and then with the pair $(y,\,x+z)$, expanding by \eqref{A-vdm-bin}:
\[
  \sum_{k}G_{n,k}(x+y)\binom{z}{k}
  =\bigl((x+y)+z\bigr)^{n}
  =\sum_{j}G_{n,j}(y)\binom{x+z}{j}
  =\sum_{j}G_{n,j}(y)\sum_{k}\binom{x}{\,j-k\,}\binom{z}{k},
\]
and comparing the coefficients of $\binom{z}{k}$ gives \eqref{AFGnkx2}.

\smallskip
\noindent\emph{(vi) Convolution.} Since
$\bigl(e^{xt}(e^{t}-1)^{k}\bigr)\bigl(e^{yt}(e^{t}-1)^{j}\bigr)=e^{(x+y)t}(e^{t}-1)^{k+j}$,
the Cauchy product gives \eqref{Con-Gnkx}.
\end{proof}

\begin{proof}[Proof of Proposition~\ref{prop:GJKH}]
Setting $x=0$ in \eqref{Exp-Gnkx} gives
\[
  G(n,k)=\sum_{\nu=0}^{k}(-1)^{\,k-\nu}\binom{k}{\nu}\nu^{n}=k!\,S(n,k),
\]
the classical formula for the Stirling numbers of the second kind
\cite[Chapter~8]{Ch}, \cite[Chapter~6]{GKP}; this is \eqref{defGnk}, and \eqref{defKnk} follows by
setting $x=0$ in \eqref{R-KnrGnr}.

For the remaining two evaluations we compute $G_{n,k}(1)$. Replacing $k$ by $k+1$ in \eqref{stipol}
at $x=0$ gives $\sum_{m}(k+1)!\,S(m,k+1)\frac{t^{m}}{m!}=(e^{t}-1)^{k+1}$; differentiating with
respect to $t$ and dividing by $k+1$ yields
\[
  \sum_{n\ge0}k!\,S(n+1,k+1)\frac{t^{n}}{n!}=e^{t}\bigl(e^{t}-1\bigr)^{k}
  =\sum_{n\ge0}G_{n,k}(1)\frac{t^{n}}{n!},
\]
the last equality being \eqref{stipol} at $x=1$. Hence $G_{n,k}(1)=k!\,S(n+1,k+1)$, and setting
$x=0$ in \eqref{R-HnrGnr} and in \eqref{R-JnrGnr} gives \eqref{defHnk} and \eqref{defJnk}
respectively.
\end{proof}

%\printnomenclature
%\printindex

\end{document}